\pdfoutput=1

\documentclass[12pt]{colt2016_mod} 

\usepackage{epsf}
\usepackage{verbatim}
\usepackage{makeidx}
\usepackage{graphicx}
\usepackage{latexsym}
\usepackage{amsbsy}
\usepackage{todonotes}
\usepackage{mathtools}
\usepackage{bbm}

\newcommand{\commentout}[1]{}

\newcommand{\m}{\setminus}

\DeclarePairedDelimiter\floor{\lfloor}{\rfloor}

\newcommand{\strong}[1]{\ensuremath{\langle #1 \rangle_{\text{\rm strong}}}}

\newcommand{\cE}{\ensuremath{\mathcal E}}
\newcommand{\cF}{\ensuremath{\mathcal F}}
\newcommand{\cG}{\ensuremath{\mathcal G}}
\newcommand{\cH}{\ensuremath{\mathcal H}}

\newcommand{\cZ}{\ensuremath{\mathcal Z}}

\newcommand{\reals}{{\mathbb R}}

\newcommand{\prob}{\ensuremath P}

\newcommand{\Exp}{\ensuremath{\text{\rm E}}}

\newcommand{\volume}{\ensuremath{V_{\text{p}}}}

\usepackage[shortlabels]{enumitem}
\usepackage{MnSymbol}
\usepackage{xcolor}
\usepackage{hyperref}
\hypersetup{colorlinks=true}
\usepackage{oubraces}
\usepackage{nicefrac} 

\usepackage[draft,multiuser]{fixme}
\fxusetheme{colorsig}
\FXRegisterAuthor{peter}{Peter}{\color{blue}PG}
\FXRegisterAuthor{niche}{Niche}{\color{red}NM}
\newcommand{\notedone}[1]{}

\newcommand{\asversion}[1]{}

\newcommand{\general}{{\textsc{general}}}
\renewcommand{\strong}{{\textsc{strong}}}

\DeclareMathOperator*{\esssup}{ess\,sup}

\newcommand{\stochleq}{\leqclosed}

\makeatletter
\DeclareRobustCommand{\qed}{%
  \ifmmode 
  \else \leavevmode\unskip\penalty9999 \hbox{}\nobreak\hfill
  \fi
  \quad\hbox{\qedsymbol}}
\newcommand{\qedsymbol}{\BlackBox}
\newenvironment{Proof}[1][\proofname]{\par
  \normalfont
  \topsep6\p@\@plus6\p@ \trivlist
  \item[\hskip\labelsep\bfseries
    #1]\ignorespaces
}{%
  \qed\endtrivlist
}
\newcommand{\proofname}{Proof}
\makeatother

 \coltauthor{\Name{Peter D. Gr\"unwald} \Email{pdg@cwi.nl}\\
 \addr Centrum Wiskunde \& Informatica (CWI) and Leiden University \\ 
 \Name{Muriel F. Pérez-Ortiz} \Email{muriel.perez@cwi.nl}\\
 \addr Centrum Wiskunde \& Informatica (CWI) \\
 \Name{Zakaria Mhammedi} \Email{mhammedi@mit.edu} \\
 \addr Massachusetts Institute of Technology (MIT)
 }

\newcommand*{\calF}{\mathcal{F}}
\newcommand*{\calG}{\mathcal{G}}
\newcommand*{\calH}{\mathcal{H}}

\newcommand*{\bracks}[1]{\left\{#1\right\}}
\newcommand*{\sqbrack}[1]{\left[#1\right]}
\newcommand*{\paren}[1]{\left(#1\right)}

\newcommand*{\norm}[1]{\left\lVert#1\right\rVert}

\newcommand*{\rme}{\mathrm{e}}
\newcommand*{\rmd}{\mathrm{d}}

\begin{document}
\hypersetup{colorlinks=true,citecolor=blue,linkcolor=blue}
\title{Exponential Stochastic Inequality}

\maketitle


%


\newcommand{\ind}[1]{\mathop{{\bf 1}_{\{#1\}}}}
\newcommand{\E}{\operatorname{\mathbf{E}}}
\newcommand{\A}{\operatorname{\mathbf{A}}}
\newcommand{\Expann}[1]{\Exp^{\textsc{ann}(#1)}}
\newcommand{\Exphel}[1]{\Exp^{\textsc{hel}(#1)}}
\newcommand{\KL}{\text{\sc KL}}
\newcommand{\Hell}{\text{\sc H}}
\newcommand{\Renyi}{\text{\sc R}}

\newcommand\tabelT{\rule{0pt}{3ex}}
\newcommand\tabelB{\rule[-1.2ex]{0pt}{0pt}}
\hyphenation{half-space hypo-thesis}

\newcommand{\bigmid}{\mathrel{}\middle|\mathrel{}}

\newcommand{\Var}{\operatorname{\mathsf{Var}}}
\newcommand{\Cov}{\operatorname{\mathsf{Cov}}}

\renewcommand{\volume}{\mathrm{vol}}

\newcommand{\peter}[1]{{\ \\ \color{blue} Peter: #1}}
\newcommand{\peterfoot}[1]{\footnote{{\color{blue} Peter: #1}}}
\newcommand{\nishant}[1]{{\ \\ \color{red} Nishant: #1}}
\newcommand{\nishantfoot}[1]{\footnote{{\color{red} Nishant: #1}}}

\begin{abstract}
   We develop the concept of {\em exponential stochastic inequality} (ESI), a novel notation that simultaneously captures high-probability and in-expectation statements. It is especially well suited to succinctly state, prove, and reason about excess-risk and generalization bounds in statistical learning; specifically, but not restricted to, the PAC-Bayesian type. We show that the ESI satisfies transitivity and other properties which allow us to use it like standard, nonstochastic inequalities.  We substantially extend  the original definition from \cite{koolen2016combining} and show that general ESIs satisfy a host of useful additional properties, including a novel Markov-like inequality. We show how ESIs relate to, and clarify, PAC-Bayesian bounds, subcentered subgamma random variables and {\em fast-rate conditions\/} such as the central and Bernstein conditions. We also show how the ideas can be extended to random scaling factors (learning rates).  
\end{abstract}

\section{Introduction}\label{sec:informal_esi}
Let $X,Y$ be two random variables. 
For fixed $\eta >0$, we define 
\begin{equation}\label{eq:esi_quick_def}
 X\stochleq_\eta Y \text{ if and only if }  \E[\rme^{\eta(X-Y)}]\leq 1.
\end{equation}
If $X \stochleq_\eta Y$ we say that {\em $X$ is stochastically exponentially smaller than $Y$}, and we call a statement of the form 
$X \stochleq_\eta Y$ an {\em Exponential Stochastic Inequality\/} or {\em ESI\/} (pronounce as ``easy''). 

The ESI is a useful tool to express certain nonasymptotic probabilistic concentration inequalities as well as generalization and excess risk bounds in statistical learning, especially but not exclusively of the {\em PAC-Bayesian\/} kind---it allows theorems to be stated more succinctly and their proofs to be simultaneously streamlined, clarified and shortened. 
This is enabled by the ESI's two main characteristics: first, 
the ESI simultaneously expresses that random variables are ordered both in expectation and with high probability---consequences of  Jensen's and Markov's inequality, respectively. Indeed, if $X \stochleq_\eta Y$ then both
\begin{equation}\label{eq:firstequation}
\text{(a)}\ {\bf E}[X] \leq {\bf E}[Y]
\text{\ and (b), with probability at least $1-\delta$},  X \leq Y + \frac{\log (1/\delta)}{\eta},
\end{equation}
for all $0 < \delta \leq 1$---this is formalized more generally in Proposition~\ref{prop:esi_characterization}. 
These simultaneous inequalities are in contrast with
  considering either ordering in probability or in expectation separately: it is
  easy to construct random variables that are ordered in expectation but not
  with high probability and vice versa. 
  The second main characteristic of the ESI is that it satisfies a useful transitivity-like property. As shown in Section~\ref{sec:esisums} below, if separately and with
  high probability $X\leq Y$ and $Y\leq Z$, the common technique of  applying the union bound to
  obtain a high-probability statement for $X\leq Z$ would lead to slightly worse bounds
  than using ESI transitivity.
ESI notation was originally introduced by \cite{koolen2016combining} and \cite{GrunwaldM20} (the first arXiv version of which came out in 2016) to improve precisely such chained bounds and to avoid stating essentially the same statement twice, once in probability and once in expectation---both statements were highly relevant in the context of the latter article. A third reason was that the bounds from \cite{GrunwaldM20} often involved annealed expectations (normalized log-moment generating functions, see the next section), and writing them out explicitly would require unwieldy nested statements like ${\bf E}[\exp ( \eta {\bf E}(\exp(\eta (...))))]\leq 1$, as can be found in for instance the pioneering work of \cite{Zhang06b}. ESI notation makes such expressions much more readable by expressing the outer expectation as an ESI, and the inner one as an annealed expectation (as defined in the next section). 
The ESI was later used in several follow-up articles \citep{mhammedi_pac-bayes_2019,GrunwaldM19,GrunwaldSZ21}, but its properties were never spelled out fully or in much detail. 

  This article gives a detailed development of the ESI. We extend its definition and notation to cover many more cases, making a novel distinction between ``weak'' and ``strong'' ESI. We provide a list of useful  properties---a calculus as it were---that can be used for manipulating ESIs. Our purpose is twofold: first, we want to showcase the ease and advantages of working with the ESI; second, we derive some new technical results---that are very conveniently expressed using the ESI---that provide a characterization of classical {\em subcentered random variables that are subgamma on the right\/} (which have been well studied before, e.g.~\cite{boucheron_concentration_2013}) and  of the main {\em fast-rate conditions\/} in statistical learning theory, the {\em Bernstein\/} and {\em central\/} conditions, extending results of \Citet{erven2015fast} to unbounded random variables. We find that such conditions only require exponential-moment control on one tail; only minimal control---of the first and second moments--- is needed for the other tail. 
\paragraph{Remark: ESI, Annealed expectation and log-moment generating function}
Of course, it has always been  common to abbreviate notations for moment and cumulant moment generating functions in order to get more compact representations and proofs of concentration inequalities. For example, the classic work of \cite{boucheron_concentration_2013} uses $\psi_X(\eta) = \log \E[\rme^{\eta X}]$ for the cumulant moment generating function. Instead of this, we use ESI and, as will become useful later, the annealed expectation (\ref{eq:annealed}) $\A^\eta[X] = {\eta}^{-1} \log\E[\rme^{\eta X}]$, 
i.e.~$\psi_X(\eta) = \eta \A^\eta(X)$. We stress that we do not claim that our notations are inherently better or more useful. Rather, we think that in some contexts uses of unnormalized $\psi_X(\eta)$ together with high-probability statements may be preferable; in other---especially related to excess- and generalization risk bounds---the normalized version $\A^{\eta}[X]$ and the ESIs are more convenient. These new notations are meant to complement, not replace, the existing. 
\subsection{Overview}
In the remainder of this introduction, we give a brief overview of what is to come, starting with the generalized definition of ESI. As a running example, we use the derivation of stochastic bounds on averages of i.i.d.~random variables. We say that $u: \reals^+ \rightarrow \reals^+$ is an  {\em ESI function\/} if it is  continuous, nondecreasing, and strictly positive. 
\begin{definition}[ESI]
    Let $u$ be an ESI function $u$
    %
as defined above. 
    We define 
\begin{equation}\label{eq:esi_full_def}
 X\stochleq_u Y \text{ if and only if for all $\epsilon > 0$,}\   \E[\rme^{u(\epsilon) \cdot (X-Y)}]\leq \rme^{u(\epsilon) \cdot \epsilon}. 
\end{equation}
\end{definition}
This definition entails that, using the original ESI notation \eqref{eq:esi_quick_def}, for all $\epsilon > 0$, if $\eta =  u(\epsilon)$, then $X \stochleq_{\eta} Y + \epsilon$. Henceforth, we shall refer to the original type of ESI in \eqref{eq:firstequation} as {\em strong ESI\/} and to the new form in \eqref{eq:esi_full_def}  as  {\em general ESI} or, if no confusion can arise, simply as {\em ESI}. The strong ESI is a special instance of the ESI, as can be seen by taking the constant function $u(\epsilon) \equiv \eta$ in (\ref{eq:esi_full_def}). In the special case that $\lim_{\epsilon \downarrow 0} u(\epsilon) = 0$, we shall refer to $X \stochleq_u Y$ as a {\em weak\/} ESI. The main reason for introducing a general ESI in \eqref{eq:esi_full_def} is that it allows us to extend most  major useful properties of the strong ESI to the weak ESI, which provides a weaker exponential right-tail control than the strong ESI and thus hold more often in practice.
We will consistently use Greek letters (usually $\eta$) to refer to constants, i.e.~strong ESIs, and Latin letters (usually $u$) to refer to functions, i.e.~general ESIs.

We now give an informal overview of some of the basic properties and implications of ESI (we present the formal statements of these properties in Section~\ref{sec:basic}). 

\paragraph{Transitivity, summation and averaging.}
As we mentioned earlier, a key property of the strong ESI is its transitivity-like property, which leads to sharper bounds than those obtained through the union bound. This property is a consequence of the fact that strong ESIs are preserved under summation, and general ESIs under averaging (Section~\ref{sec:esisums}, Proposition~\ref{prop:esi_sums}, Corollary~\ref{cor:esi_averages}). 
To demonstrate the latter property, let $\{ X_f: f \in \cF \}$ be a family of random variables and let $X_{f,1}, \ldots, X_{f,n}$ be i.i.d.~copies of each $X_f$. Suppose we are given the ESIs 
\begin{equation}\label{eq:weakesiintro}
X_{f,i} \stochleq_u 0 
\text{ for all } f \in \cF\text{ and } i \in [n].
\end{equation} 
Then, we can conclude via Corollary~\ref{cor:esi_averages}, for all $f \in \cF$, that 
\begin{equation}\label{eq:friday}
\frac{1}{n} \sum_{i=1}^n X_{f,i} \stochleq_{n\cdot u} 0. 
\end{equation}
This does not only imply that $\E[\sum X_{f,i}] \leq 0$, but also the high-probability statement that for all $0 < \delta \leq 1$, 
\begin{align}\label{eq:demo}
    \frac{1}{n} \sum_{i=1}^n X_{f,i} \leq \inf_{\epsilon > 0} \ \left( \epsilon + \frac{\log (1/\delta)}{n \cdot u(\epsilon)} \ \right). 
\end{align}
Additionally, the ESI  (\ref{eq:weakesiintro}) implies that all the moments of the right tail of each $X_{f,i}$ are finite. Under the quite weak condition that the $X_{f,i}$ also have uniformly bounded second moment on the left tail, we can infer via Proposition~\ref{prop:newversionsubgammamain} in Section~\ref{sec:weakesi}---under the assumption that (\ref{eq:weakesiintro}) holds for some common ESI function $u$---that they also satisfy a (weak) ESI for a function $u(\epsilon) = C^* \epsilon \wedge \eta^*$ for some $C^*, \eta^* >0$. Thus, without loss of generality, we can take a $u$ that is linear near the origin. We can then see that for large enough $n$, the minimum in (\ref{eq:demo}) is achieved at an $\epsilon$ with $u(\epsilon) = C^*\epsilon < \eta^*$. In that case, the infimum can be computed through differentiation and (\ref{eq:demo}) becomes
\begin{align}\label{eq:demob}
    \frac{1}{n}  \sum_{i=1}^n X_{f,i} \leq c \cdot \left( \frac{\log (1/\delta)}{n} \right)^{\alpha} 
\end{align}
for some $c > 0$ and  $\alpha = 1/2$, a standard bound in statistical learning theory \citep{vapnik_statistical_1998,shalevshwartz2014understanding}. 
In Section~\ref{sec:weakesi} (Proposition~\ref{prop:newversionsubgammamain}), we give a number of equivalent characterizations of the general ESI in terms of {\em subcentered, subgamma random variables\/} of which the result that “$u$ can be taken  linear near the origin” is just one instance. 

\paragraph{From weak to strong ESI: excess risk bounds.}
The transitivity property also allows us to prove fast rates of convergence of empirical averages to their expected value.
As we will detail in the sequel, this is of particular interest for proving excess risk bounds of machine learning algorithms. Now, we consider  $\{X_f: f \in \cF\}$ that all satisfy the ESI $X_f \stochleq_u 0$ for a common ESI function $u$ of the form 
\begin{equation}\label{eq:gammagamma}
u(\epsilon) = C^*  \epsilon^{\gamma} \wedge \eta^*
\text{\ for some $0 \leq \gamma \leq 1$ and $C^*, \eta^* > 0$}.
\end{equation}
Again, for large enough $n$, the minimum in (\ref{eq:demo}) is achieved at an $\epsilon$ with $u(\epsilon)  < \eta^*$, and differentiation gives that (\ref{eq:demob}) now holds with $\alpha = 1/(1+ \gamma)$.
If $\gamma < 1$, we say that the average satisfies a {\em fast-rate\/} statement. 

To see why, we briefly need to explain one of the most important applications of the ESI, namely, providing {\em excess-risk bounds\/} in statistical learning theory \citep{Zhang06a,Zhang06b,GrunwaldM20}. Here, we assume that there is an underlying sequence of i.i.d.~{\em data\/} $Z_1, \ldots, Z_n$, each $Z_i$ having the same distribution as $Z$. Each $f \in \cF$ represents a {\em predictor\/}, and there is a loss function $\ell_f(Z) \in \reals$ quantifying the loss that the predictor $f$ makes on $Z$. Often, $Z$ 
is of the form $Z= (U,Y)$, and $f$ represents a function mapping covariates (or features) $U$ to $Y \subset \reals$. An example of this setup is regression with the  squared error loss $\ell_f((U,Y))= \frac{1}{2}(Y- f(U))^2$. One can fit other prediction and inference problems such as classification and density estimation into this framework as well \Citep{erven2015fast,GrunwaldM20}. We now define the {\em excess loss\/} that the predictor $f$ makes on the outcome $Z$ as 
$L_{f}  = L_f(Z) = \ell_f(Z) - \ell_{f^*}(Z)$ where $f^*$ is the minimizer of $f\mapsto\E[\ell_f(Z)]$ over $Z$; for simplicity, we assume $f^*$ to exist and to be unique. Thus, $L_{f}$ measures how much better or worse $f$ performs compared to the theoretically optimal $f^*$ on a particular $Z$.
Based on a sample $Z^n = (Z_1, \ldots, Z_n)$, learning algorithms output an “estimate” or “learned predictor” $\hat{f}:= \hat{f}|Z^n$, the latter notation indicating the dependence of $\hat{f}$ on $Z^n$. Sometimes, e.g.~in Bayesian and PAC-Bayesian inference (see below), they output, more generally, a learned distribution $\hat{\Pi} = \hat{\Pi}|Z^n$ on $f \in \cF$.  The goal is to design an algorithm  whose {\em excess risk\/} ${\bf E}_{Z \sim P} [L_{\hat{f}|Z^n}(Z)]$ (or ${\bf E}_{\bar{f} \sim \hat\Pi} {\bf E}_{Z \sim P} [L_{\bar{f}|Z^n}(Z)]$ if the algorithm outputs a distribution) converges to zero fast, with high probability and/or in expectation. To this end, it is crucial to control how fast the {\em empirical excess risk\/} $n^{-1} \sum_{i=1}^n L_{f,i}$
(where $L_{f,i}  = \ell_f(Z_i) - \ell_{f^*}(Z_i)$)
of each fixed $f \in \cF$ converges to its expectation $\E[L_f]$.
In practice, in simple cases (e.g.~bounded losses) the collection of negative excess risks $\{X_f: f \in \cF \}$ with $X_f = - L_f$ satisfies a weak ESI, so that (\ref{eq:demob}) holds with $\alpha = 1/2$---in line with what one might expect from the central limit theorem.  
However, in many interesting cases (e.g.~bounded squared error loss), something better
(larger $\alpha$) can be attained, because  (\ref{eq:demo}) holds, for all $f \in \cF$, with $u(\epsilon) = C^* \epsilon^{\gamma} \wedge \eta^*$ for a $\gamma< 1$ (in the specific case of bounded squared error loss it even holds with $\gamma=0$). Then (\ref{eq:demob}) implies that, for each individual $f$, $n^{-1 } \sum_{i=1}^n L_{f,i} = O(n^{-\alpha})$ with $\alpha = 1/(1+ \gamma)$, and this usually translates into learning algorithms that also converge at this fast (i.e., faster than $1/\sqrt{n}$, since $\gamma > 0$) rate; an example for empirical risk minimization (ERM) is given in the sequel. 

Using different terminology and notation (not ESI), \Citet{erven2015fast} already identified that collections $\{L_f:f \in \cF \}$ such that all $X_f = -L_f$ satisfy $X_f \stochleq_u 0$ for $u(\epsilon) = C^* \epsilon^{\gamma} \wedge \eta^*$ (as above) allow for fast rates; in their terminology, such a family satisfies the {\em $u$-central fast-rate condition}. They showed  that, for bounded loss functions (and hence uniformly bounded $L_f$), satisfying this property for some $\gamma$ is equivalent to $\cF$ satisfying the celebrated $\beta$-{\em Bernstein\/} condition, with $\beta = 1- \gamma$. The Bernstein condition \citep{audibert2004pac,bartlett_empirical_2006,audibert_fast_2009} is a more standard, well-known condition for fast rates. \Citet{erven2015fast} left open the nagging question whether the Bernstein and central fast-rate conditions remain equivalent for unbounded loss functions. As one of the main results in this article, we show in Theorem~\ref{thm:strongesi} (Section~\ref{sec:strongesi}) that this is indeed the case as long as the left tail of the excess risk is exponentially small, and the right tail satisfies a mild condition on its second moment.  

\paragraph{PAC-Bayesian bounds.}
The ESI is particularly well suited to PAC-Bayesian analysis. To demonstrate this, we continue to assume that there are i.i.d.~$Z_1, \ldots, Z_n$ such that, for all $f \in \cF$, $X_{f,i}= g_{f}(Z_i)$, that is, $X_{f,i}$ can be written as a function of $Z_i$ for some function $g_f$ which may, but does need to be a negative excess loss (in fact, in many applications it will be an expected loss minus an absolute, non-excess empirical loss; see e.g.~\cite{GrunwaldSZ21}). We can easily combine the ESIs as (\ref{eq:weakesiintro}) into a statement that simultaneously involves all $f \in \cF$ by using {\em PAC-Bayesian\/} bounds  \Citep[see][]{catoni_pac-bayesian_2007,mcallester_pac-bayesian_1998,van_erven_pac-bayes_2014,Guedj19,Alquier21}. As we show in Section~\ref{sec:pacbayes}, in ESI notation such bounds take a simple form, and become easy to manipulate and combine. By Part~2 of Proposition~\ref{prop:pacbayes} and (\ref{eq:friday}) we immediately get the ESI
\begin{equation}\label{eq:weakesiintrob}
\E_{\bar{f} \sim \hat{\Pi}} \left[\frac{1}{n} \sum_{i=1}^n X_{\bar{f},i}\right] \stochleq_{nu} \frac{\KL(\hat{\Pi} \| \Pi_0)}{n u},
\end{equation}
where for random variables $X,Y$, and function $u$ the ESI $X \stochleq_{u} Y/u$ means $\E[\rme^{u(\epsilon)\cdot  X-Y}]\leq \rme^{u(\epsilon) \cdot \epsilon}, \forall \epsilon>0$. In Eq.~\eqref{eq:weakesiintrob}, $\KL$ is the Kullback-Leibler divergence; $\Pi_0$ is a distribution on $\cF$ called a “prior” in analogy to prior distributions in Bayesian statistics; and $\hat{\Pi}$ is allowed to be any distribution on $\cF$ that may depend on data $Z^n$ and that represents the learning algorithm of interest. (If we write $\E$ without subscript, we refer to the expectation of $Z$ and hence to that of $X_f$; with subscript $\bar{f} \sim \hat{\Pi}$, the expectation is taken over $\hat{\Pi}$.) In simple cases, $\hat{\Pi}$ will be a degenerate distribution with mass one on an estimator (learning algorithm) $\hat{f} = \hat{f}{|Z^n}$, as above, and $\Pi_0$ will have a probability mass function $\pi_0$ on a countable subset of $\cF$,  and then ${\KL(\hat{\Pi} \| \Pi_0)} = - \log \pi_0(\hat{f})$.
\commentout{
Again, for regular $\{X_f\}_{f \in \cF}$ we know that we can take $u(\epsilon) = C^* \epsilon \wedge \eta^*$ which then implies the in-probability statement 
\begin{align}\label{eq:democ}
  \E_{\bar{f} \sim \hat{\Pi}}\left[ \frac{1}{n} \sum_{i=1}^n X_{\bar{f},i}\right] \leq C_1 \cdot \left( \frac{{\KL(\hat{\Pi} \| \Pi_0)} + \log (1/\delta)}{n} \right)^{1/2} +
C_2 \cdot \left( \frac{{\KL(\hat{\Pi} \| \Pi_0)} + \log (1/\delta)}{n} \right),
\end{align}
which is a simple version of a standard PAC-Bayesian bound. 
With a final additional definition/abbreviation, we can now also  succinctly derive and express one of the deepest PAC-Bayesian excess risk bound, namely the bound of \cite{Zhang06a,Zhang06b} (originally given in both in-probability and in-expectation forms) as worked out further by \cite{GrunwaldM20}: for $\eta > 0$, we define the {\em annealed expectation\/} as
\begin{equation}\label{eq:annealed}
  \A^\eta[X] = \frac{1}{\eta}\log\E[\rme^{\eta X}].
\end{equation}
The annealed expectation is a rescaling of the cumulant generating function, ``a well known provider of nonasymptotic bounds''\citep{catoni_pac-bayesian_2007}; we remark that in some other works, ``annealed expectation of $X$'' refers to what is  $- \A^{\eta}[-X]$ in our notation. 
With  definition (\ref{eq:annealed}), Zhang's bound, given now simultaneously in its expectation and in-probability version, becomes simply
\begin{equation}\label{eq:weakesiintroZhang}
\E_{\bar{f} \sim \hat{\Pi}} \left[\frac{1}{n} \sum_{i=1}^n X_{\bar{f},i} - \A^{u}[X_{\bar{f}}] \right] \stochleq_{nu} \frac{\KL(\hat{\Pi} \| \Pi_0)}{n (u/2)},
\end{equation}
(see the next section for notation $\A^u$) which follows as an immediate consequence of the PAC-Bayes Proposition~\ref{prop:pacbayes}, Part 3. 
Further consequences of (\ref{eq:weakesiintrob}) are readily inferred}
Now Lemma~\ref{lem:witness} in Section~\ref{sec:strongesi}, which is adapted from \cite{GrunwaldM20} but now receiving a very different interpretation in the present ESI context, shows that, if the ESI (\ref{eq:weakesiintro}) holds with $u(\epsilon) = C^* \epsilon^{\gamma} \wedge \eta^*$ as in (\ref{eq:gammagamma}) (providing right-tail control of the $X_f$), then under a weak additional condition on the left tail, the so-called {\em witness condition}, there exists a constant $c > 0$ such that, for all $f \in \cF$, $i \in [n]$,
\begin{equation}
    \label{eq:democ}
    X_{f,i} - c \E[X_{f,i}] \stochleq_{u/2} 0
\end{equation}
Note that (\ref{eq:democ}) is not a trivial consequence of (\ref{eq:weakesiintro}) because we have $\E[X_{f,i}] \leq 0$. 
Using again Corollary~\ref{cor:esi_averages} (about ESI averages), Part 2 of Proposition~\ref{prop:pacbayes} (PAC-Bayes), and (\ref{eq:democ}), we immediately get the ESI
\begin{equation*}
\E_{\bar{f} \sim \hat{\Pi}} \left[\frac{1}{n} \sum_{i=1}^n X_{\bar{f},i}- c \E[X_{\bar{f}}] \right] \stochleq_{n u/2}   \frac{2 \KL(\hat{\Pi} \| \Pi_0)}{n u },
\end{equation*}
which, barring suboptimal constant factors, coincides with the main excess risk bound of \cite{GrunwaldM20}. 
Indeed, in the case with $L_f = - X_f$, an excess loss, the above can be rewritten as 
\begin{equation*}
c \E_{\bar{f} \sim \hat{\Pi}}  \E_{Z \sim P} [L_{\bar{f}}] \stochleq_{n u/2}
 \E_{\bar{f} \sim \hat{\Pi}} \left[\frac{1}{n} \sum_{i=1}^n L_{\bar{f},i} \right] +
\frac{2 \KL(\hat{\Pi} \| \Pi_0)}{nu},
\end{equation*}
which provides an excess risk bound for the learning algorithm embodied by $\hat{\Pi}$. It says that the expected performance on future data---if we use the randomized predictor obtained by sampling from $\hat{\Pi}$---is in expectation as good as it performed on the sample $Z^n$ itself, up to a $\KL/n$ complexity term. If $\hat{\Pi}$ implements empirical risk minimization, placing mass $1$ on the $\hat{f}\in \cF$ that minimizes the loss on $Z^n$, then the empirical excess loss $\E_{\bar{f} \sim \hat{\Pi}} \left[\frac{1}{n} \sum_{i=1}^n L_{\bar{f},i} \right] = \frac{1}{n} \sum_{i=1}^n L_{\hat{f},i}$  must be $\leq 0$; if  further $\cF$ is finite and $\Pi_0$ is uniform on $\cF$, this implies, following a minimization analogous to (\ref{eq:demo}) but now with $ \log (1/\delta) + 2 \KL(\hat{\Pi} \| \Pi_0)$ in the numerator, that  depending on $\gamma$ as in (\ref{eq:gammagamma}), a rate of $O((\log |\cF|)^{1/(1+ \gamma)})$ is achieved  both in expectation and in probability.  \cite{GrunwaldM20} show variations of this bound (with discretized infinite $\cF$) to be minimax optimal in some situations.

\paragraph{Further developments: partial order, ESI Markov, Random $\eta$, non-i.i.d.}
Besides the properties needed for the above-illustrated applications to fast-rate, PAC-Bayesian, excess-risk bounds, we provide some further properties of the ESI that are of general interest. We start in Section~\ref{sec:basic} with basic properties of the ESI, including an extensive treatment of transitivity. We show that the strong ESI formally defines a {\em partial order\/} relation. We also provide answers to natural questions such as “does the ESI characterization (\ref{eq:esi_full_def}) admit a converse?” and we show that ESIs imply some other curious stochastic inequalities. In particular, we show an {\em ESI Markov inequality\/}, which we find intriguing---whether it will prove useful in applications remains to be seen, though.

Section~\ref{sec:whenesi} gives detailed characterization of strong and general ESI, and contains, besides new notation, also some truly novel results. Section~\ref{sec:pacbayes} revamps existing results to provide the connection to PAC-Bayes (the main result of Section~\ref{sec:pacbayes}---Proposition~\ref{prop:pacbayes}---was already illustrated above). While strictly speaking not containing anything new, it reorganizes and disentangles existing PAC-Bayesian proof techniques, showing that there really are at least three inherently different basic PAC-Bayesian results that are used as building blocks in other works.  Section~\ref{sec:random} contains some new results again, concerning the situation where the $\eta$ in strong ESIs is not fixed but itself a random, i.e.~data-dependent variable. The article ends with Section~\ref{sec:non-iid-sequences} that extends ESIs to the non-i.i.d.~case, connecting them  to random processes, showing that ESIs defined on a sequence of random variables remain valid under optional stopping. 
Example~\ref{ex:waldmeetszhang} in that section lays out an intriguing connection between Zhang's PAC-Bayesian inequality and the Wald identity, a classic result in sequential analysis \citep{Wald:1947}. All longer proofs are deferred to appendices.

\section{Basic ESI Properties}\label{sec:basic}
In this section, we show the properties of the ESI that were anticipated in the introduction. We start with Section~\ref{sec:preliminaries}, where we lay down the notation that will be used in the rest of the article; in particular, for the annealed expectation. In Section~\ref{sec:basic_properties}, we show basic properties of the ESI. There, we show the main implications of a random variable satisfying an ESI, and layout useful properties that will be used in the next sections. In Section~\ref{sec:esi_converse} we show a partial converse to definition of the ESI: if a random variable has a subexponential right tail, it satisfies an ESI---we show a more definitive converse in Section~\ref{sec:whenesi}. In Section~\ref{sec:esisums} we show the main properties of the ESI in relation to its transitivity and its use to bounding sums of independent random variables. In Section~\ref{sec:partial_order}, we show that the ESI defines a partial order on random variables. We end with Section~\ref{sec:curious_markov} with a curiosity, a Markov-like inequality that replaces the requirement of positivity in Markov's inequality with the weaker $0\stochleq_\eta X$.

\subsection{Preliminaries: additional definitions and notation\label{sec:preliminaries}}

Throughout the article, we fix some probability space $(\Omega,\Sigma,P)$. Whenever we speak of random variables or a class of random variables without indicating their  distribution, we assume that they are all defined relative to this triple, and that their expectation is well-defined. To be more precise, we call a  function $X: \Omega \rightarrow \reals$ a random variable  if it is measurable; we may have ${\bf E}[X_+] = \infty$ (then $\E[X] = \infty$) or ${\bf E}[X_-]= \infty$ (then ${\bf E}[X] = - \infty$), but not both. Here and in the sequel, ${\bf E}$ denotes expectation under $P$ and $X_+ = 0 \vee X; X_- =  0 \vee (-X)$. 
\begin{definition}[Subcentered and regular]
    We call a random variable $X$ {\em subcentered\/} if ${\bf E}[X] \leq 0$ and  {\em regular\/} if ${\bf E}[X^2] < \infty$. We call a family of random variables $\{X_f: f \in \cF \}$ regular if $\sup_{f \in \cF} {\bf E}[X_f^2] < \infty$.
\end{definition}
 The reason for reserving the grand word “regular” for this simple property  is that, as we will see in Section~\ref{sec:whenesi},  as long as it holds everything works out nicely;  in particular, we obtain an equivalence between random variables satisfying an ESI being{\em subcentered, uniformly subgamma random variables}.

\begin{definition}[Annealed expectation]
     Let $\eta > 0$ and let $X$ be a random variable. We define the {\em annealed expectation\/} as
\begin{equation}\label{eq:annealed}
  \A^\eta[X] = \frac{1}{\eta}\log\E[\rme^{\eta X}].
\end{equation}
\end{definition}
The annealed expectation is a rescaling of the cumulant generating function, ``a well-known provider of nonasymptotic bounds''\citep{catoni_pac-bayesian_2007}; we remark that in some other works, ``annealed expectation of $X$'' refers to what is  $- \A^{\eta}[-X]$ in our notation.  Of course, the definition of the ESI could have been written using the annealed expectation as
  \begin{align}\label{eq:esi2}
    X \stochleq_{\eta} Y  \ \ \text{ if and only if } \ \ \A^\eta[X-Y]\leq 0 .
  \end{align}
We need one more, final extension of the ESI notation. Let $u$ be an ESI function---a continuous, positive, increasing function. For any random variables $X$ and $Y$ and function $f: \reals^+ \times \reals \rightarrow \reals$, we write 
\begin{equation}
    \label{eq:extranotation}
    X\stochleq_u  f(u,Y) \text{\ as shorthand for: for all $\epsilon > 0$, with $\eta = u(\epsilon)$, }
\E[\rme^{\eta (X-f(\eta,Y))}] \leq e^{ \eta \epsilon}.
\end{equation}
Notice that we already used this notation implicitly in (\ref{eq:weakesiintrob}). 
\commentout{  
 Second, we define $X\stochleq^*_\eta Y$ to highlight the
  special case in which the right-hand side of \eqref{eq:esi2} is satisfied with
  equality.
The reason to consider the special case $\stochleq^*$ where (\ref{eq:esi2})
holds with equality is that it yields the best bounds in terms of additive
constants in the following sense. It is trivially true that for any random
variable $X$
\begin{equation*}
  X\stochleq_\eta^*\A^\eta[X];
\end{equation*}
this only expresses the fact that
$\A^\eta[X] - \A^\eta[X] = \A^\eta[X - \A^\eta[X]] = 0$. However, when
$X\stochleq_\eta 0$, the same statement `strengthens' the premise because it
shows that a ES-negative random variable can be further bounded (in the ES
sense) by a negative constant, namely, $\A^\eta[X]$. We present this observation
in the following Proposition.
\begin{proposition}\label{prop:zhangify}
  If $X \stochleq_{\eta} 0$, then $ X \stochleq^*_{\eta} \A^{\eta}[X].$
\end{proposition}
}

\subsection{Basic Properties of the ESI \label{sec:basic_properties}}

In the following proposition, we state the main consequences of two random variables $X,Y$ satisfying an ESI $X\stochleq_\eta Y$; namely, that they are ordered both in expectation and with high probability. In the next section we give a partial converse to this definition: if two random variables $X,Y$ are ordered with high probability, they satisfy an ESI with modified constants. A more definitive characterization is the subject of Section~\ref{sec:whenesi}.

\begin{proposition}[ESI characterization]
  \label{prop:esi_characterization}
  Let $X,Y$ be two random variables such that $X\stochleq_u Y$ for some ESI function $u$. Then
  \begin{enumerate}
  \item \label{itm:main_esi_negative_expectation} $\E[X]\leq \E[Y]$. If  $u \equiv \eta$ is constant (strong ESI), then the
    inequality is strict unless $X=Y$ a.s.
  \item \label{itm:main_esi_probability statement} $X$ and $Y$ are ordered with
    high probability, that is, for all $\epsilon > 0$,
    $\prob(X\geq Y + \epsilon + K)\leq \rme^{-u(\epsilon) K }$, or equivalently, for
    any $\delta\in[0,1]$
    \begin{equation}\label{eq:main_weak_esi_high_probability}
      X\leq Y + \inf_{\epsilon > 0} \left( \frac{1}{u(\epsilon)}\log\frac{1}{\delta} + \epsilon \right),
    \end{equation}
    with probability higher than $1-\delta$.
    In the special case of $u \equiv \eta$ constant, i.e.~a strong ESI,    $\prob(X\geq Y  + K)\leq \rme^{- \eta K }$  or, for
    any $0 < \delta \leq 1$,
    \begin{equation}\label{eq:main_esi_high_probability}
      X\leq Y + \frac{1}{\eta}\log\frac{1}{\delta},
    \end{equation}
    with probability higher than $1-\delta$.
  \end{enumerate}
  \end{proposition}
\begin{Proof}
 Jensen's inequality and
  the fact that the function $x\mapsto \rme^{-\eta x}$ is strictly convex yields Part 1 (including strictness for the strong ESI case). 
  For Part 2, apply Markov's inequality to
  $e^{u(\epsilon) (X-Y - \epsilon)}$ to give $P(X \geq Y + \epsilon + (\log (1/\delta)/u(\epsilon)) ) \leq \delta$. Since this holds simultaneously for all $\delta > 0$, the result follows.  
  \end{Proof}
For simplicity, we did not spell out the consequences of an ESI of the form $X \stochleq_u f(u,Y)$ as defined above in (\ref{eq:extranotation}); the extension of Proposition~\ref{prop:esi_characterization} to this case is entirely straightforward. 
\paragraph{Remark}
If the ESI $X \stochleq_u Y$ is not strong, then it is possible that the inequality in Part 1 of the proposition is not strict, i.e.~that  $\E[X]= \E[Y]$. An example is given by $P(X=1) = P(X=-1) = 1/2$, $P(Y=0)= 1$. By the $\cosh$ inequality we have $X \stochleq_u Y$ for $u(\epsilon) = \epsilon/2$, yet obviously $\E[X] = \E[Y]$.

We now introduce some very basic useful properties of ESIs that we will freely use in the remainder of the article. 
\begin{proposition}[Useful Properties]\label{prop:useful-properties-esi}
  Let $X,Y,Z$ be three random variables and let $u$ and $u^*$ be ESI functions.
  The following hold:
  \begin{enumerate}
  \item \label{item:esi-properties-leq} If $X\stochleq_u Y$ and $Y\leq Z$ almost surely then $X\stochleq_u Z$.
  \item \label{item:esi-properties-leq-converse} $X\leq Y$ almost surely if and
    only if $X\stochleq_\eta Y$ (strong ESI) for every $\eta>0$.
  \item\label{item:esi-properties-convexity} If $X\stochleq_{u^*} Y$, then
    $X\stochleq_{u^{\circ}} Y$ for each ESI function $u^{\circ}$ with $u^{\circ}\leq u^*$ (by which we mean: for all $\epsilon > 0, u^{\circ}(\epsilon) \leq u^*(\epsilon)$).
    \item\label{item:esi-properties-positive-part}
    Suppose that $Z \stochleq_{u} 0$. Then $Z_+ - \E[Z_+] \leq Z_+ \stochleq_{u} (\log 2)/u$ and similarly, for every $c > 0$, we have $Z \ind{Z \geq c} \leq Z_+ \stochleq_{u} (\log 2)/u$.  
  \item\label{item:esi-properties-mean-relation}  For $\eta > 0$, it holds that
      \begin{equation}\label{eq:annealedpower}
X-\A^{\eta}[X] \stochleq_{\eta} 0.
    \end{equation}
and hence 
    \begin{equation}
      \label{eq:genrenc} \E[X]\leq \A^\eta[X].
    \end{equation}
  \end{enumerate}
\end{proposition}
\begin{proof}
We only give the proofs for strong ESIs with constant $u$; the generalizations to general ESI functions $u=\eta$ are immediate. 
  For \ref{item:esi-properties-leq}, notice that if $Y\leq Z$, then
  $X-Y\geq X-Z$. This in turn implies $0\geq \A^\eta[X-Y]\geq \A^\eta[X-Z]$ so
  that $X\stochleq_\eta Z$. For \ref{item:esi-properties-leq-converse} it is
  clear that if $X-Y\leq 0$, then $\A^\eta[X-Y]\leq 0$ for each $\eta$. For the
  converse recall that if the $p-$norm
  $\norm{X}_p = (\E[|X|^p])^{1/p}$ of a random variable $X$ is finite for all $p>0$, then, as $p\to\infty$, 
  $\norm{X}_p \to \esssup |X|$, the essential supremum\footnote{The essential
    supremum of a random variable $X$ is the smallest constant $c$ such that
    $X\leq c$ almost surely.} of $X$. Note that by assumption
  $\A^\eta[X-Y] = \log\norm{\rme^{X-Y}}_\eta\leq 0$ for all $\eta>0$, and thus
  taking $\eta\to\infty$ we can conclude that
  $\log (\esssup \rme^{X-Y}) \leq 0$, that is, $X-Y\leq 0$ almost surely.
  \ref{item:esi-properties-convexity} follows from the convexity of the function
  $x\mapsto \rme^{\eta x}$. 
  \ref{item:esi-properties-positive-part} follows since 
  \begin{equation}\label{eq:zplus}
 \E[\rme^{\eta Z}] = \E[\rme^{\eta Z_+}] + \E[\rme^{- \eta Z_-}] - 1,
  \end{equation}
  so that
  $$
  \E \left[\rme^{\eta ( Z_+ + (\log 2)/\eta)} \right] = \frac{1}{2}
  \E \left[\rme^{\eta Z_+} \right] \leq 
  \frac{1}{2} \left( 
  \E \left[\rme^{\eta Z} \right] +1 \right) \leq 1,  
  $$
  where the final inequality follows by assumption. 
  \ref{item:esi-properties-mean-relation} follows from Jensen's inequality and (\ref{eq:annealedpower}) is just definition chasing.
\end{proof}

\subsection{A partial converse to the basic ESI characterization \label{sec:esi_converse}}
\begin{proposition}\label{prop:exponential_tail_esi}
  Let $Z$ be a random variable. If there exist $a,b>0$ such that
  \begin{equation}\label{eq:prob_exponential_difference}
    P(Z \geq  \epsilon) \leq a\rme^{-b \epsilon}
  \end{equation}
  for each $\epsilon>0$, then, for each $0 <  \eta' < b$, there is a constant
  $c> 0$ such that   $ Z \stochleq_{\eta'}  c$, where 
  \begin{equation}\label{eq:esi_bounded_difference}
 c= \frac{1}{\eta'}\log\paren{1+\frac{a\eta'}{b-\eta'}}.
  \end{equation}
  In particular, if for some $\eta$ the precise statement (\ref{eq:main_esi_high_probability}) holds for all $0 < \delta \leq 1$ with probability at least $1-\delta$, then by taking $a=1$, $b=\eta$, $\eta'= \eta/2$, $Z= X-Y$,  we find that
  $X\stochleq_{\eta/2} Y + (2/\eta) \log 2$.\\
\end{proposition}
This proposition shows that if we have an exponentially small right-tail probability for $Z$, then an ESI statement with a $C^* > 0$ on the right must already hold; in particular, if we weaken an ESI to its high-probability implication and then convert back to an ESI, we loose both a factor of $2$ in the scale factor $\eta$ and an additive constant.  If we can additionally assume that $\E[Z] \leq 0$, then  both main ESI implications from Proposition~\ref{prop:esi_characterization} 
hold and indeed, if  additionally $Z$ is regular---if its second moment is bounded---, we  get a more complete converse of Proposition~\ref{prop:esi_characterization} (allowing ESI functions $u$ rather than just fixed  $\eta$); this is done in Proposition~\ref{prop:newversionsubgammamain} later on.

\subsection{Sums of random variables and transitivity}\label{sec:esisums}

 In this subsection we show how ESIs are  useful when proving probabilistic bounds for sums $\sum_{i = 1}^n X_i$ of random variables---not necessarily independent---, and how this leads to a transitivity-like property. All our results are stated, and valid for, strong ESIs; in Corollary~\ref{cor:esi_averages} we look at averages rather than sums and, as stated there, the results become valid for general ESIs.

Thus, consider the sum $S_n= \sum_{i=1}^nX_n $. In the case that
strong ESI bounds are available for each of them individually, that is, when
$X_i\stochleq_{\eta_i} 0$ for some $\eta_i>0$ and $i=1,\dots,n$, then we seek to
obtain a similar statement for $S_n$---in analogy to the sum of negative numbers
remaining negative. In order for $S_n$ to remain negative with large
probability, independence or, more generally,  association assumptions need to be made. We discuss this fact after the statement of the bounds. A set of
random variables $X_1,\dots,X_n$ is said to be negatively associated
\citep[cf.][]{joag-dev_negative_1983, dubhashi_balls_1998} if for any two
disjoint index sets $I,J\subset\bracks{1,\dots,n}$ it holds that
$\Cov(f(X_i,i\in I), g(X_j, j\in J))\leq 0$, or more succinctly, if 
\begin{equation*}
  \E[f(X_i,i\in I)g(X_j, j\in J)]\leq \E[f(X_i,i\in I)]\E[g(X_j, j\in J)]
\end{equation*}
for any choice of monotone increasing functions\footnote{We mean that the
  functions are increasing in each argument when the others are held fixed.} $f$
and $g$. Examples of negatively associated random variables include
independent random variables, but also include negatively correlated jointly
Gaussian random variables, and permutation distributions. The following proposition can be obtained. 
\begin{proposition}
  \label{prop:esi_sums}
  Let $X_1,\dots,X_n$ be random variables such that $X_i\stochleq_{\eta_i}0$ for
  some $\eta_1,\dots,\eta_n>0$. Then
  \begin{enumerate}
  \item \label{item:esi_sum_dependent} Under no additional assumptions, $S_n\stochleq_\eta 0$ with $\eta =
    \paren{\sum_{i=1}^n\frac{1}{\eta_i}}^{-1}$.
  \item \label{item:esi_sum_independent} If $X_1,\dots,X_n$ are negatively associated---in particular, if they are independent---, then $S_n\stochleq_\eta 0$ with
    $\eta = \min_i \eta_i$.
\end{enumerate}

\end{proposition}
\begin{proof}
  We prove the case $n=2$; its generalization is straightforward. Note that
  $\A^\eta[X] = \log ||\rme^{X}||_\eta$, where $||\cdot||_\eta$ denotes the
  $p$-norm at $p=\eta$ given by $||Y||_\eta = \paren{\E|Y|^{\eta}}^{1/\eta}$.
  Using H\"older's inequality we get
  \begin{equation}\label{eq:esi_hoelder}
    \A^\eta[X_1 + X_2] \leq \A^{\eta p}[X_1] + \A^{\eta q}[X_2],
  \end{equation}
  where $p, q \geq 1$ are Hölder conjugates related by $p^{-1} + q^{-1} = 1$.
  Replacing $p = 1 + \frac{\eta_1}{\eta_2}$ and $\eta$ as in \ref{item:esi_sum_dependent}, the result follows. For Part 
  \ref{item:esi_sum_independent}, note that for independent or negatively
  associated random variables it holds that
  $ \A^{\eta}[S_n] \leq \sum_{i=1}^n \A^{\eta}[X_i]\leq 0$ with $\eta = \min_i \eta_i$, from which the
  result follows.
\end{proof}
With an eye towards the PAC-Bayesian bounds anticipated in the introduction, we now present a corollary of the previous proposition which holds for averages instead of sums. Its proof is omitted as it is a direct application of the previous proposition. Under this modification, the results hold for arbitrary ESI functions $u$ instead of constants $\eta$; thus, it is this corollary that allows for the ESI treatment of PAC-Bayesian bounds. As above, consider random variables $X_1, \ldots, X_n$ and let $\bar{X} = n^{-1} S_n$ be their average. We obtain:
\begin{corollary}
    \label{cor:esi_averages}
     Suppose that $X_i \stochleq_{u_i}$ for ESI functions $u_1,\ldots, u_n$. Then
  \begin{enumerate}
  \item \label{item:esi_average_dependent} Under no additional assumptions, $\bar{X} \stochleq_{n u} 0$ with $u =
    \paren{\sum_{i=1}^n\frac{1}{u_i}}^{-1}$.
  \item \label{item:esi_average _independent} If $X_1,\dots,X_n$ are i.i.d.~and $u= u_1 = u_2 = \ldots = u_n$, then $\bar{X} \stochleq_{n u} 0$. 
\end{enumerate}
\end{corollary}

The results obtained in Part  \ref{item:esi_sum_dependent} and
\ref{item:esi_sum_independent} of the Proposition~\ref{prop:esi_sums} above have very different quantitative
consequences because of the difference in their association assumptions. In the case that for some fixed $\eta>0$ it holds that
$X_i\stochleq_\eta 0$ for $i=1,\dots,n$, then Proposition \ref{prop:esi_sums}
implies that $S_n\stochleq_{\eta/n}0$. Through Proposition
\ref{prop:esi_characterization} this in turn implies that with probability
higher than $1-\delta$ it holds that
\begin{equation*}
  S_n \leq \frac{n}{\eta}\log\frac{1}{\delta}.
\end{equation*}
This does not rule out the possibility that, even if all of the $X_i$ are with
large probability negative, their sum might still grow linearly with the number
of terms $n$---for instance under complete dependency, when all $X_i= X_1$. On the other hand, when
$X_i,\dots,X_n$ are independent or negatively associated, this cannot be the
case. Indeed, Proposition \ref{prop:esi_sums} implies $S_n\stochleq_\eta 0$ which after
using again Proposition \ref{prop:esi_characterization}, implies that with
probability higher than $1-\delta$
\begin{equation*}
  S_n\leq \frac{1}{\eta}\log\frac{1}{\delta}.
\end{equation*}


As a corollary, the anticipated property that is reminiscent of transitivity holds for
$\stochleq_\eta$. 
\begin{corollary}{\bf \ [Transitivity]} \label{cor:esi_transitivity}
  If $X\stochleq_{\eta_1} Y$ and $Y\stochleq_{\eta_2}Z$, then
  \begin{enumerate}
  \item $X\stochleq_\eta Z$ with $\eta = (1/\eta_1 + 1/\eta_2)^{-1}$.
  \item If $X,Y$ and $Z$ are negatively associated, then
    $X\stochleq_\eta Z$ with $\eta = \min\bracks{\eta_1,\eta_2}$.
  \end{enumerate}
\end{corollary}
\begin{proof}
  Use that $X-Z = (X-Y) + (Y-Z)$ and Proposition \ref{prop:esi_sums}.
\end{proof}

We close this subsection with an observation about the common practice of using probabilistic union bounds. Even though in general the union bound is tight, in the presence of ESIs it is loose. 

\begin{remark}[Chaining ESI bounds improves on  union
  bound]\label{rmk:esi-beats-union} {\rm 
  Suppose $X$,$Y$,$Z$ are random variables such that $X\stochleq_\eta Y$, and $Y\stochleq_\eta Z$. For each $a>0$,
  Proposition \ref{prop:esi_characterization} implies both that
  $\prob(X\geq Y + a)\leq \rme^{-\eta a} $ and that
  $P(Y\geq Z + a) \leq \rme^{-\eta a}$. Using directly the union bound on these
  two events, one would obtain that $\prob(X\geq Z + 2a)\leq 2\rme^{-\eta a}$,
  or equivalently that with probability higher than $1-\delta$
  \begin{equation}
    X\leq Z + \frac{2}{\eta}\log \frac{2}{\delta}
  \end{equation}
  while using  Proposition \ref{cor:esi_transitivity} one obtains that
  $X\stochleq_{\eta/2}Z$, which, again using Proposition \ref{prop:esi_characterization} implies
  that with probability higher than $1-\delta$
  \begin{equation}
    X\leq Z + \frac{2}{\eta}\log\frac{1}{\delta}.
  \end{equation}
  This is better than the previous bound because of the factor appearing in the
  logarithm. This seems like a minor difference, but the effect adds up when
  chaining $n$ inequalities of this type. Indeed, in that case one obtains (by
  using ESI) in-probability bounds that tighter than the union bound by a
  $\log n$ factor.}
\end{remark}

\subsection{ESI as a stochastic ordering \label{sec:partial_order}}
ESIs are different from standard ordering relations in that they depend on the parameter $u$. We may view them as such standard ordering relations simply by adding existential quantifiers. Thus we may set 
\begin{align*}
  &  X \stochleq_{\general} Y \text{\ if and only if there exists an ESI function $u$ s.t.\ }
    X \stochleq_{u} Y \\
  &  X \stochleq_{\strong} Y \text{\ if and only if there exists $\eta^* \in \reals^+$ s.t.\ }
    X \stochleq_{\eta^*} Y   
\end{align*}
\begin{proposition}\label{prop:esi_stochastic_ordering}
  Let $\{X_f: f \in \cF \}$ be a set of random variables. Then   $\stochleq_{\strong}$ defines a partial order on this set.  
\end{proposition}
We note that $\stochleq_{\general}$ does not define a partial order. Indeed, if  $P(X=1) = P(X=-1) = 1/2$ and $P(Y=0)=1$ we have, as a consequence of a small computation, both $X \stochleq_{\general}Y $ and $Y \stochleq_{\general} X$. However, $X \neq Y$ a.s.
\begin{proof}[Proposition~\ref{prop:esi_stochastic_ordering}]
We need to check whether the order is reflexive, transitive and antisymmetric. Reflexivity is immediate, transitivity follows from Corollary~\ref{cor:esi_transitivity} above, and antisymmetry from Proposition~\ref{prop:esi_characterization}, Part 1.
\end{proof}

In light of this proposition, it might be of interest to compare this partial order to the usual order of stochastic dominance, and its generalization, $k$th order stochastic dominance.

\subsection{ESI-positive random variables: a curious Markov-like inequality\label{sec:curious_markov}}

In this section we deal with random variables $X$ that are positive in the strong ESI
sense, that is, $0\stochleq_\eta X$ for some $\eta>0$. Notice that by
Proposition \ref{prop:esi_characterization}, we know that for each $a>0$, we can
bound the probability that $X$ is smaller than $-a$---a left-tail bound---by
$P(X\leq -a)\leq \rme^{-a}$. Additionally, we can obtain a Markov-style
inequality for the probability that $X$ is large---a right-tail bound.

\begin{proposition}\label{prop:markov}
  Let $X$ be a random variable such that $0\stochleq_\eta X$. Then, for any $a>0$,
  \begin{equation*}
    \prob(X\geq a)
    \leq
    \frac{\E[X]}{a} + \frac{p\log (1/p)}{\eta a}
    \leq
    \frac{\E[X]}{a} + \frac{1}{\rme\eta a},
  \end{equation*}
  where $p = \prob(X < 0)$
\end{proposition}
\begin{remark}{\rm  
  Notice that the first inequality reduces to Markov's inequality in the case
  that $p=P(X<0) = 0$, that is, when $X$ is a nonnegative random variable, the requirement for the standard Markov's inequality to hold. Thus, the intuition behind the proposition is that, since $0 \stochleq_{\eta} X$ expresses  $X$ is “highly likely almost positive”, it allows us to get something close to Markov after all.
  }
\end{remark}
Notice that for any increasing real-valued function $f$ it holds that
\begin{equation*}
\prob(X\geq a) = \prob(f(X) \geq f(a))
\end{equation*}
and consequently if $f(X)$ is positive in the ESI sense, that is,
$0\stochleq_\eta f(X)$ for some $\eta>0$, our version of Markov's inequality can
be used in the same spirit in which Chebyshev's inequality follows from Markov's
inequality.

\begin{corollary}
  If $f$ is increasing and $X$ is a random variable such that
  $0\stochleq_\eta f(X)$, then
  \begin{equation}
    \prob(X\geq a)
    \leq
    \frac{\E[f(X)]}{f(a)} + \frac{p\log (1/p)}{\eta f(a)}
    \leq
    \frac{\E[f(X)]}{f(a)} + \frac{1}{\rme\eta f(a)}
  \end{equation}
  where $p = \prob(f(X) < 0)$.
\end{corollary}

\section{When does a family of RVs satisfy an ESI?}
\label{sec:whenesi}

In this section, we show a converse to the definition of the ESI. A special role will be payed by regular, subgamma, subcentered random variabes. As we will see, subgamma makes reference to random variables whose (right tail) is lighter than that of a gamma distribution. Recall from Section~\ref{sec:basic} that we call a family of random variables regular if its second moment is uniformly bounded; subcentered, if their expectation is negative.

\subsection{General ESIs and subcentered subgamma random variables}
\label{sec:weakesi}
We say that a random variable $X$ has a {\em  $(c,v)$-subgamma right tail \/} if it satisfies
\begin{equation}\label{eq:sub_gamma_first_description}
  X - \E[X]\stochleq_\eta \frac{1}{2}\frac{v\eta}{1-c\eta}
\end{equation}
for some $c,v>0$ and  all $\eta$ with $0\leq c\eta\leq 1$  
\citep[see][Section 2.4]{boucheron_concentration_2013}. This name is in relation to
the fact that random variables that are gamma distributed satisfy it.
Subgamma random variables are well-studied: \Citet{van_de_geer_bernsteinorlicz_2013} studied empirical
processes of random variables that satisfy a tail condition implied by
(\ref{eq:sub_gamma_first_description}). Sufficient conditions for (possibly
unbounded) random variables to satisfy a subgamma bound have been known for a
long time \citep[cf.][p. 202-204]{uspensky_introduction_1937}. This topic has been also treated by \Citet[Section 2.2.2]{van_der_vaart_weak_1996} and by \citet[Section 2.8]{boucheron_concentration_2013}.

The following proposition shows that for a regular family, that is, a family satisfying  $\sup_{f \in \cF} \E[X^2_f] < \infty$, ESI families---families that satisfy $X_f\stochleq_u 0$ for all $f$ and some $u$---can be equivalently characterized in a number of ways. Its most important implications are that a regular family of random variables satisfies an ESI, i.e.~for all $f \in \cF$,  $X_f \stochleq_u 0$,  
\begin{enumerate}
    \item[(a)] if and only if its elements are all subcentered and  uniformly subgamma on the right, and
    \item[(b)] if and only if it satisfies an ESI for a function $h$ that is linear near $0$.
\end{enumerate}
We also note that the first converse that we presented  to the main ESI implications, Proposition~\ref{prop:exponential_tail_esi}, was still relatively weak, in the sense that if we have an ESI of the form $Z \stochleq_{u} 0$, we apply the central Proposition~\ref{prop:esi_characterization} to calculate that for all $\epsilon> 0$,  (a) $P(Z \geq K+\epsilon) \leq e^{-u(\epsilon)K}$ and (b) $\E[Z] \leq 0$, and we “back-transform”  (a) to an ESI via the converse in Proposition~\ref{prop:exponential_tail_esi} (which only uses (a)), we obtain $Z \stochleq_{u'} c$ for some ESI function $u'$ and some $c > 0$, i.e.~we loose a additive constant term.  With the help of the proposition below, we can use (a) jointly with (b)  to conclude (using 6. below) that $Z \stochleq_{u'} 0$ for an ESI function $u'$, i.e.~we can “back-transform” without loosing any additive terms in the ESI. 
\begin{proposition}\label{prop:newversionsubgammamain}
Let $\{X_f \}_{f \in \cF}$ be a regular family, i.e.~$\sup_{f \in \cF} \E[X^2_f] < \infty$. Then, the following statements are equivalent:
\begin{enumerate}
\item There is an ESI function $u$ such that for all $f \in \cF$, $X_f \stochleq_{u} 0$.
\item There is a constant $C^*> 0$ and a constant $\eta^* >0$ such that, uniformly over all $f \in \cF$, 
$X_f \leq  X_f - \E[X_f] \stochleq_{\eta^*} C^*$.
\item There exist $c,v > 0$ such that, for all $f\in\calF$, the $X_f$ are subcentered and  have a $(c,v)$-subgamma right tail.
\item There is an ESI function $h$ such that, for all $f \in \cF$, we have $X_f \leq X_f - \E[X_f]  \stochleq_{h} 0$ where $h$ is of the form 
$h(\epsilon) = C \epsilon \wedge \eta^*$. 
\item There exists $c,v > 0$ such that, for all $f\in\calF$, the $X_f$ are subcentered and, for each $f \in \cF$ and $0 < \delta \leq $1, with probability at least $1- \delta$,
\begin{equation}\label{eq:boucheronsubgamma}
X_f \leq \sqrt{2 v \log (1/\delta)} + c \log (1/\delta).  
\end{equation}
\item There exists $a > 0$ and  a differentiable function $h: \reals^+_0 \rightarrow \reals^+_0$
with $h(\epsilon) > 0$, $h'(\epsilon) \geq 0$ for $\epsilon > 0$, such that for all $f \in \cF$,  the $X_f$ are subcentered and $P(X \geq \epsilon) \leq a \exp(- h(\epsilon))$ (in particular $h$ may be a positive constant or a linear function of $\epsilon$). 
\end{enumerate}
\end{proposition}
In Appendix~\ref{app:weakesi}, we state and prove an extended version of this result, Proposition~\ref{prop:newversionsubgammaapp}, which shows that if (3) holds for some pair $(c,v)$, then (5) holds for the same $(c,v)$ and (4) holds for $\eta^* = 1/(2c)$ and $C = 1/(2v)$.
It also shows that regularity is only required for some of the implications between the four statements above. In particular, it is not needed for $(3) \Rightarrow (4)$, $(3) \Rightarrow (5)$ and $(4) \Rightarrow (1)$, and for $(1) \Rightarrow (2)$; a strictly weaker condition---control of the first rather then second moment of the $X_f$---is sufficient. 
However, 
Example~\ref{ex:varianceneeded} below shows that,
in general, some sort of minimal control of the supremum of the second moment, and hence of the left tails of the $X_f$, is needed (note though that higher moments of $|X_f|$ need not exist) to get $(2) \Rightarrow (3)$ and hence the full range of equivalences. 
Indeed, the only difficult part in the proposition above is the implication $(2) \Rightarrow (3)$. It is a direct consequence of  Theorem~\ref{thm:muriel} below (again proved in Appendix~\ref{app:weakesi}), which shows that
we can actually directly relate the constants $(c,v)$ in “right subgamma” to the constants $C^*$ and $\eta^*$. The proof extends an argument from \cite[Theorem 2.10]{boucheron_concentration_2013}.
\begin{theorem}\label{thm:muriel}\label{prop:centered_bounded_esi}
Let $U$ be a random variable such that $U - \E[U] \stochleq_{\eta^*} C$ for  some fixed
     constants $C$ and $\eta^*>0$. Then for $0 < \eta \leq \eta^*$, we have $U- \E[U]\stochleq_\eta \frac{1}{2}\frac{v\eta}{1-c\eta}$ for $v= \Var[U]+ 2\exp(\eta^* C)$ and $c = 1/\eta^*$.
\end{theorem}
\begin{example}\label{ex:varianceneeded}
{\rm Let $U$ be a random variable with, for $U \leq -1$, density $p(u) = 1/|u|^{\nu}$ for some $\nu$ with $5/2 < \nu <3$.
Then $P(U \leq -1) = \int_{- \infty}^{-1} p(u) = 1/(\nu -1)$. We set $x_{\nu} = (\nu-1)/(\nu-2)^2$ and $P(U= x_{\nu}) = 1- P(U \leq -1)$ so that $P(- 1 < U < x_{\nu})=  P(U > x_{\nu}) = 0$. Then $\E[U \cdot \ind{U \leq -1}]= - 1/(\nu-2)$ and hence $\E[U] = 0$, and an easy calculation shows that $U= U- \E[U] \stochleq_1 C^*$ with $C^* = \log (1 + \exp(x_\nu))$. Hence the premise inside (3) of Proposition~\ref{prop:newversionsubgammamain} is satisfied for family $\{ U\}$, but  $\Var[U] = \infty$ so that $\{U \}$ is not regular so that the general precondition of  Proposition~\ref{prop:newversionsubgammamain}
does not hold. And indeed (proof in Appendix~\ref{app:weakesi}) we find that $(\E[\exp(\eta U)] -1 )/\eta^2 \rightarrow \infty$ as $\eta \downarrow 0$, showing that the right-subgamma property is not satisfied. }
\end{example}

\subsection{Interpolating between weak and strong ESIs}

We may think of a weak and a strong ESI as two extremes in a hierarchy of possible tail bounds---the strong ESI given the lightest tails; the weak, the heaviest. We now define ESI families and $\gamma$-strong ESI family, where $\gamma\in [0,1]$ is the interpolating factor.  

\label{sec:strongesi}
\commentout{
\newcommand{\sql}{\ensuremath{\text{\sc sql}}}
Let the function $\sql$ be given by 
$$
\sql(x) = \begin{cases} x^2 & \text{\ if $|x| \leq 1$} \\ |x| & \text{\ if $|x| > 1$}.
\end{cases}
$$
$\sql$ is {\em s\/}ymmetric, {\em q\/}uadratic near the origin and linear for large $|x|$. Any other continuous function with these properties --- such as e.g.~$\log \cosh$ --- could be used as well in the definition and theorem below. 
}
\begin{definition}
We say that a family of random variables $\{X_f: f \in \cF \}$ is an {\em  ESI family} if there exists an ESI function $u$ such that for all $f \in \cF$, $X_f \stochleq_{u} 0$.
For $0 \leq \gamma \leq 1$, we say that the family is a $\gamma$-strong ESI family if there exist $C^* > 0, \eta^* > 0$ and a function $u(\epsilon) = C^* \epsilon^{\gamma} \wedge \eta^*$ such that for all $f \in \cF$, $X_f \stochleq_u 0$.  
For an interval $I \subseteq [0,1]$, we say that the family is an $I$-strong ESI family if for all $\gamma \in I$, it is a $\gamma$-strong ESI family. 
\end{definition}
Note that if for some $\eta > 0$, all $X_f$ satisfy the strong ESI $X_f \stochleq_\eta 0$, then in this terminology they form a $0$-strong ESI family. 
\begin{proposition}\label{prop:etaepsilon}
    Fix $\gamma \in [0,1]$. A regular family $\{X_f: f \in \cF \}$  is a $\gamma$-strong ESI family if and only if there exists $C^{\circ} > 0, 0 < \eta^{\circ} < 1$  such that for all $f \in \cF$, 
    \begin{equation}\label{eq:etaepsilon}
\text{for all $0 < \eta \leq \eta^{\circ}$:\ }    X_f \stochleq_{\eta} C^{\circ} \eta^{\frac{1}{\gamma}} 
        \end{equation}
\end{proposition}
where we set $\eta^{1/0} := \lim_{\gamma \downarrow 0} \eta^{1/\gamma} = 0$.
\begin{proof}
Let $u(\epsilon) = C^* \epsilon^\gamma \wedge \eta^*$ as in the definition of $\gamma$-strong. 
Set $\epsilon^* > 0$ to be such that $C^* e^{*\gamma} = \eta^*$, i.e.~the value of $\epsilon$ at which $u(\epsilon)$ starts to become a horizontal line. 
By definition, we have 
$$
(\ref{eq:etaepsilon})  \Leftrightarrow \text{$\forall  \eta \in (0, \eta^{\circ}]$:\ } \E[\rme^{\eta X_f}] \leq e^{\eta   \cdot C^{\circ} \eta^{1/\gamma}} \text{\ and\ } X_f \stochleq_u 0 \Leftrightarrow \text{$\forall \epsilon \in (0,\epsilon^*]$:\ } \E[\rme^{C^* \epsilon^{\gamma} X_f}] \leq e^{C^* \epsilon^{\gamma} \cdot \epsilon}
$$
If we set $C^{\circ} = 1/ C^{*\gamma}$ and for each $\epsilon \in (0,\epsilon^*]$, we set  $\eta = C^* \epsilon^{\gamma}$ then both  expressions coincide for each such $\epsilon$ and for each $\eta \in (0,\eta^*]$; the result follows. 
\end{proof}
The importance and motivation of $\gamma$-strong ESI families
comes from their application in fast-rate results as already indicated in the introduction. As there, let $\{L_f: f \in \cF \}$ be a collection of excess-loss random variables, $L_f$ being the excess loss of predictor $f$, and let $X_f = -L_f$ be the negative excess loss. Then $\{X_f: f \in \cF \}$ being a $\gamma$-strong ESI family coincides with, under the definitions of \Citet{erven2015fast}, $\cF$ satisfying the {\em $u$-central fast rate condition\/} for $u(\epsilon) = C^* \epsilon^\gamma \wedge \eta^*$. 
They showed that, for bounded loss functions (implying that the $L_f$ are uniformly bounded), under the $u$-central fast-rate condition with $u$ as above, and with a suitable notion of complexity $\textsc{comp}$, one can get an excess risk rate of order $O((\textsc{comp}/n)^{1/(1+\gamma)}$, as was illustrated for the special case of ERM with finite $\cF$ in the introduction. 
\cite{GrunwaldM20} (GM from now on) extended their result to the case that the $L_f$ are unbounded, and only have minimal tail control on the right tail, the tail satisfying a condition they called the {\em witness-of-badness\/} or just {\em witness\/} condition. They showed that both this condition and a $u$-central fast-rate condition hold in many practically interesting learning situations. We state the  witness-of-badness condition here in terms  of $X_f = -L_f$ rather than $L_f$, since it can then also be used for collections $\{X_f \}_{f \in \cF}$ that simply satisfy an ESI and have no excess-loss interpretation. 
\begin{definition}{\bf \ [Witness-of-Badness Condition]}
\label{def:witness-badness}
There  exists $0 < c < 1$ and 
$C > 0$ such that for all $f \in \cF$, 
\begin{equation}\label{eq:witness}
\E[(- X_f) \ind{- X_f \geq C}] \leq c {\bf E}[- X_f].
\end{equation}
\end{definition}
Note that this condition only makes sense for random variables with $\E[X_f] \leq 0$ (which automatically holds if $X_f \stochleq_u 0$). It then automatically holds whenever the $X_f$ have uniformly bounded left tail; GM show that it holds in many other cases as well, with the caveat that the constant $C$ often scales linearly in the (suitably defined) dimension, making the resulting bounds not always optimal in terms of this dimension.  

GM's Lemma 21, translated into ESI notation, now says the following:
\begin{lemma}{\bf [GM's Lemma 21, rephrased as ESI]}
\label{lem:witness}
    Suppose that $\{ X_f: f \in \cF \}$ is an ESI family, i.e.~$X_f \stochleq_u 0$,  such that  $\sup_{\epsilon > 0} u(\epsilon) < \infty$ (in particular, any ESI family can be expressed as such if it is regular) and suppose that the witness-of-badness condition as above holds. Then, there is a $c^* > 0$ such that, for all $f \in \cF$, 
    \begin{equation}\label{eq:thereturnofthewitness}
    X_f - c^* \E[X_f] \stochleq_{u/2} 0. 
    \end{equation}
\end{lemma}
GM go out of their way to optimize for the constant $c^*$; our interest being in the big picture here, we will not provide details about the constant. 
It can be seen from the proof that their result does not rely on the $L_f$ having an interpretation as excess risks: it holds for general families $\{X_f: f \in \cF \}$.

It was already discussed in the introduction how GM's Lemma (Lemma~\ref{lem:witness}) can lead to fast rates. Essentially, to design learning algorithms that attain fast rates within this framework one needs  that $\{X_f \}_{f \in \cF}$ is a $\gamma$-strong ESI family for $\gamma < 1$, which gives exponential control of the $X_f$'s right tail, ensuring that empirical losses converge to their mean faster than $1/\sqrt{n}$,  and on top of that one needs witness-of-badness, which gives a very different kind of control of the potentially heavy-tailed left tail, to ensure that this convengence also holds for empirical losses with a constant times their expectation, the empirical risk, subtracted; finally one uses a PAC-Bayesian combination of all $f \in \cF$ to get the desired excess risk bound.

\subsection{The Bernstein conditions and the $\gamma$-strong ESI}
In their general treatment of fast rate conditions, \Citet{erven2015fast} showed how the $u$-central condition for $\{L_f: f \in \cF \}$ with $u(\epsilon) = C^* \epsilon^{\gamma} \wedge \eta^*$ is equivalent to the $\beta$-Bernstein condition, with $\beta = 1-\gamma$. The $\beta$-Bernstein condition is a better known condition for obtaining fast rates in excess risk bounds \Citep[see][]{bartlett_empirical_2006,erven2015fast,audibert_fast_2009}. Their equivalence result only holds for uniformly bounded $L_f$; extending it to general---unbounded---excess risks remained a nagging open question.  In Theorem~\ref{thm:strongesi} below, we fully resolve this issue for abstract families of random variables that do not require an excess risk interpretation. As a by-product, the theorem implies an analogue to Lemma~\ref{lem:witness} that relates $\gamma$-strong ESIs to strengthenings thereof with $c \E[X]$ subtracted as in (\ref{eq:thereturnofthewitness}).

We first recall the standard definition of the Bernstein condition: 
\begin{definition}[$\beta$-Bernstein Condition]
  Let $\beta\in[0,1]$. We say that a family of random variables
  $\{L_f\}_{f\in\calF}$ satisfies the $\beta$-Bernstein condition if, for all $f \in \cF$, $\E[L_f] \geq 0$ and there is
  some $B>0$ such that
  \begin{equation}
    \label{eq:bernstein-easiness-condition}
    \text{for all $f \in \cF$,\ } \E[L_f^2]\leq B(\E[L_f])^\beta.
  \end{equation}
  The 1-Bernstein condition is also known as the {\em strong Bernstein condition}.
\end{definition}
Suppose that $\{L_f\}_{f\in\cal F}$ is a regular family. Then, it is straightforward to show that the family satisfies $\beta$-Bernstein if and only if satisfies $\beta'$-Bernstein for all $\beta' \in [0,\beta]$.
Motivated by this equivalence, we may start considering half-open intervals $[0,\beta)$ --- it turns out that this gives a version of Bernstein that is much better suited for comparing with ESI families for unbounded random variables. Formally: 
\begin{definition}\label{def:almostbernstein}
  Let $I \subseteq [0,1]$ be an interval. We say that a family of random variables  $\{L_f\}_{f\in\calF}$ satisfies the $I$-Bernstein condition if for all $f \in \cF$,  $\E[L_f] \geq 0$ and for all $\beta'\in I$, there is
  some $B>0$ such that
  \begin{equation*}
\text{for all $f \in \cF$,\ }  \E[L_f^2]\leq B(\E[L_f])^{\beta'}.
  \end{equation*}
\end{definition}
It is immediately verified that every family that satisfies the $I$-Bernstein for nonempty $I$ automatically has uniformly bounded second moment, i.e.~it is regular. 

The following theorem shows that for regular families of random variables, the notions of $(b,1]$-strong ESI and $[0,b)$-strong Bernstein coincide (with for the Bernstein condition, $X_f$ replaced by $-X_f$), under a “squared version” of the witness condition defined below; this condition has not been proposed before in the literature, as far as we know.

In Appendix~\ref{app:strongesi} we state and prove an extended version of the theorem, in which the various conditions  on $\{X_f\}_{f \in \cF}$ needed for the various implications in the theorem below are spelled out; these conditions are all implied by regularity but are in some cases weaker. 

\begin{definition}{\ \bf [Squared-Witness Condition]} 
We consider the following condition for a family of random variables $\{U_f : f \in \cF \}$: there  exists $0 < c < 1$ and 
$C > 0$ such that for all $f \in \cF$, 
\begin{equation}\label{eq:tallball}
\E[U_f^2 \ind{U_f^2 \geq C}] \leq c {\bf E}[U_f^2 ].
\end{equation}
\end{definition}
The original witness-of-badness condition (Definition~\ref{def:witness-badness}) is just (\ref{eq:tallball}) with $- X_f$ in the role of $U_f^2$, where  $- X_f$ represents, as in this section, the excess risk. Below we use the equation with $U_f^2 = X_f^2 = (- X_f)^2$ and  also with $U_f^2  = ((X_f)_{-})^2$. 

Special cases of parts of the following theorem for uniformly bounded $X_f$, for which regularity and squared witness automatically hold, were stated and  proven by \cite{koolen2016combining}, and earlier by \cite{GaillardStoltzVanErven2014}. 
\begin{theorem}\label{thm:strongesi}
Let $\{X_f: f \in \cF\}$ be a regular family of random variables that satisfies the squared-witness condition above for $U_f = X_f$ or for $U_f = (X_f)_{-}$. Then the following statements are equivalent: 
\begin{enumerate}
\item 
$\{ - X_f: f \in \cF \}$ satisfies the $[0,b)$-Bernstein condition for some $0 < b < 1$ and $\{X_f : f \in \cF\}$ is an ESI family.
\item  For all $\beta\in [0,b)$, 
for all $c \geq 0$, all $0 \leq c^* < 1$, there exists  $\eta^{\circ} >0 $ and $C^{\circ} >0$ such that for all $f \in \cF$, all $0 < \eta \leq \eta^{\circ}$,
\begin{equation}\label{eq:almostthere}
X_f  + c \cdot \eta \cdot  X^2_f - c^* \cdot  \E[X_f] \stochleq_{\eta} C^{\circ} \cdot \eta^{\frac{1}{1-\beta}},
\end{equation}
or equivalently, by Proposition~\ref{prop:etaepsilon}, there exists $\eta^*, C^* > 0$ such that 
$$X_f  + c \cdot \eta \cdot  X^2_f - c^* \cdot  \E[X_f] \stochleq_u 0,$$ where $u(\epsilon) = C^* \epsilon^{1-\beta} \wedge \eta^*$.
\item For all $\beta\in [0,b)$,  there exists $\eta^{\circ} >0$ and $C^{\circ} > 0$ such that for all $f \in \cF$, all $0 < \eta \leq \eta^{\circ}$, we have  
$$X_f  \stochleq_{\eta} C^{\circ} \eta^{\frac{1}{1-\beta}},$$ i.e.~  
$\{X_f : f \in \cF\}$  is a $(b,1]$-strong ESI family. Equivalently, by Proposition~\ref{prop:etaepsilon}, there exists $\eta^*, C^* > 0$ such that for all $f \in \cF$,  $X_f  \stochleq_u 0$, where $u(\epsilon) = C^* \epsilon^{1-\beta} \wedge \eta^*$. 
\end{enumerate}
\end{theorem}
Note that, if $\{-X_f: f \in \cF \}$ satisfies $[0,1)$-Bernstein, then ${\bf E}[X_f] \leq 0$ for all $f \in \cF$; also, $X^2_f \geq 0$. Therefore, the implication (2) $\Rightarrow$ (3) is trivial. 
The proof of Theorem~\ref{thm:strongesi} is based on Theorem~\ref{thm:strongesib} and Lemma~\ref{lem:strongESIregularity} in the appendix, which, taken together, are a bit stronger than Theorem~\ref{thm:strongesi}, which comes at the price of a more complicated statement. In a nutshell, on  one hand, the implication (1) $\Rightarrow (2)$ still holds even if the witness-type condition does not hold. On the other hand, the implication (2) $\Rightarrow$ (3) $\Rightarrow$ (1) still holds if the right-hand side of (\ref{eq:almostthere}) is replaced by $0$ (strong ESI family) and then the conclusion in (3) also becomes that (a)  $X_f \stochleq_{\eta^*} 0$ and, in (1), that (b) $\{-X_f: f \in \cF\}$ satisfies $[0,1]$-Bernstein. 
Thus, the implication (3) $\Rightarrow (2)$ can be seen as a second-order analogue of Lemma~\ref{lem:witness}, allowing not just $c^* \E[X_f]$ but also $c \eta X^2_f$ to be added to $X_f$, at the price of requiring the witness-squared rather than the standard witness condition. 

Having an ESI with $X^2$ outside of the expectation is not  needed for the excess-risk bound discussed in the introduction, but it is crucial for several other PAC-Bayesian generalization bounds that also achieve faster rates (better data-dependent bounds) if the data are sampled from a distribution such that a Bernstein condition holds \citep[see][]{mhammedi_pac-bayes_2019}.

\section{PAC-Bayes}\label{sec:pacbayes}

In this section we prove and write the PAC-Bayesian bounds
\Citep[see][]{mcallester_pac-bayesian_1998,van_erven_pac-bayes_2014,catoni_pac-bayesian_2007,Guedj19,Alquier21}
in ESI notation, under which they take a pleasant look. Importantly, we find that in the existing literature,  “applying the PAC-Bayesian” or “Donsker-Varadhan change-of-measure” technique can really mean at least three different things. Using the annealed expectation notation together with ESI can disentangle these different uses, appearing as the three different parts in Proposition~\ref{prop:pacbayes} below.  We let 
$\{X_f\}_{f\in\calF}$ again be a family of random variables. Let $\Pi$ and
$\hat{\Pi}$ be two equivalent probability measures (two probability measures
with the same null sets) on $\calF$ such that their mutual Radon-Nykodim
derivatives exist. Define the Kulback-Leibler divergence $\KL(\hat{\Pi}\|\Pi)$
as
\begin{equation*}
  \KL(\hat{\Pi}\|\Pi) = \E_{\hat{\Pi}}\sqbrack{\log\frac{\rmd \hat{\Pi}}{\rmd \Pi}}.
\end{equation*}
PAC-Bayesian theorems are based on the relation of convex duality that exists
between the Kullback-Leibler divergence and the cumulant generating function---the logarithmic moment generating function. We state them here as strong ESIs but, since the following results hold for all $\eta > 0$, it also follows that they also hold with $\eta$ replaced by  any ESI function $u$.

We continue to assume that there are i.i.d.~$Z, Z_1, \ldots, Z_n$ such that, for all $f \in \cF$, $X_f= g_f(Z)$ and $X_{f,i}= g_{f}(Z_i)$ can be written as a function of $Z$ and $Z_i$  respectively for some function $g_f$
Hence the distribution of $Z$ determines the distribution of $X_f$ and $X_{f,i}$ for all $i \in [n],f \in \cF$. In this section we need to pay special attention to the notation; $P$ is not the only measure that plays a role. We will write the relevant measure in subscript of $\E$ and ${\bf A}$. Thus,  $\E_{(Z,f)\sim P\bigotimes\Pi}[X_f] = \iint g_f(Z)\rmd \Pi(f)\rmd P(Z)$---notice that $\Pi$ might depend on $Z$. With this in mind, we state the following proposition.  

\begin{proposition}\label{prop:pacbayes}
  Let $\bracks{X_f}_{f\in\calF}$ be a family of random variables and let $\eta > 0$. Then for any two equivalent
  distributions $\Pi_0$ and $\hat{\Pi}$ on $\calF$ we have: 
  \begin{enumerate}
  \item The following ESI holds:
\begin{equation}\label{eq:newfirstpacbayes}
\E_{\bar{f} \sim \hat{\Pi}} [X_{\bar{f}}]  - \A^\eta_{(Z,\bar{f})\sim P \bigotimes \Pi_0}[X_{\bar{f}}] \stochleq_{\eta}   \frac{\KL(\hat{\Pi} \| \Pi_0)}{\eta} .
\end{equation}
  \item
  Suppose further that
  for each $f\in\calF$, $X_f \stochleq_{\eta} 0$. Then we have: 
  \begin{equation}\label{eq:oldfirstpacbayes}
    \E_{\bar{f} \sim \hat{\Pi}} [X_{\bar{f}}] \stochleq_{\eta} \frac{\KL(\hat{\Pi} \| \Pi_0)}{\eta}.
\end{equation}
\item Now let again $\bracks{X_f}_{f\in\calF}$ be an arbitrary family.
We have:
\begin{equation}\label{eq:oldsecondpacbayes}
\E_{\bar{f} \sim \hat{\Pi}} [X_{\bar{f}} - \A^\eta_{Z \sim P}[X_{\bar{f}}]] \stochleq_{\eta}   \frac{\KL(\hat{\Pi} \| \Pi_0)}{\eta}.
\end{equation}
\end{enumerate}
\end{proposition}
In case the family $\{X_f \}_{f \in \cF}$ satisfies $X_f \stochleq_{\eta} 0$ for all $f \in \cF$, then $\A^{\eta}[X_f]$ is negative and the third result is a “boosted” version of the second, and therefore should usually give stronger consequences. 
\begin{proof}
  For Part 1: the variational formula for the $\KL$ divergence (which appeared already in the
  work of \citet{gibbs_elementary_1902}) states that
  \begin{equation*}
    \log\E_{\bar{f}\sim\Pi}[\rme^{\eta X_{\bar{f}}}] \geq \eta\E_{\bar{f}\sim\hat{\Pi}}[X_{\bar{f}}]  - \KL(\hat{\Pi}\|\Pi)
  \end{equation*}
  Taking exponentials, $P$-expected value on both sides, and using
  Fubini's theorem,
  \begin{equation*}
    \E_{\bar{f}\sim\Pi}\E_{Z\sim P}[\rme^{\eta X_f}]
    \geq
    \E_{Z\sim P}\sqbrack{\exp\paren{\eta \E_{\bar{f}\sim\hat{\Pi}}[X_{\bar{f}}]  - \KL(\hat{\Pi}\|\Pi)}},
  \end{equation*}
  which is a rewriting of the result. Part 2 follows from Part 1 
 by noting that in this case we further have $1 \geq \E_{\bar{f}\sim \Pi}\E_{Z\sim P}[\rme^{\eta X_{\bar{f}}}]$. 
Part 3 follows from using Part 2 with $X_{\bar{f}}$ replaced by $X_{\bar{f}} - 
\A^{\eta}_{Z\sim P}[X_f]$; for this new random variable, ESI is guaranteed by  the simple observation (\ref{eq:annealedpower}) in Proposition~\ref{prop:useful-properties-esi} so that Part 3 follows.
\end{proof}
 The second result is the most straightforward one and has been used to derive many PAC-Bayesian results, e.g.~\cite{seldin2012pac,tolstikhin2013pac,wusplit,mhammedi_pac-bayes_2019}. 
 The third result, illustrated in Example~\ref{ex:zhang} below, has  been (implicitly) used to get PAC-Bayesian {\em excess-risk\/} bounds such as those by \cite{Zhang06a,Zhang06b} and \cite{GrunwaldM20}. 
 The first result, illustrated in Example~\ref{ex:begin}, can be used to derive a whole class of PAC-Bayesian bounds that include one of the strongest and best-known early bounds, the Langford-Seeger-Maurer bound \citep{Seeger02,langford2003pac,maurer2004,Alquier21}. It would be interesting to see how recent articles establishing bounds based on conditional mutual information (which can be thought of as an in-expectation version of a specific PAC-Bayesian bound) fit in. For example, \citet{GrunwaldSZ21} uses the second result, but this is not so clear for recent bounds such as those by \cite{HellstromD22}.
\begin{example}{ \bf\ [Zhang's Inequality]\ }{\label{ex:zhang}\rm
Zhang's inequality \citep{Zhang06a,Zhang06b} provides one of the strongest PAC-Bayesian-type excess-risk bounds in the literature; more precisely, it gives a ``proto-bound'' which can then be further specialized to a wide variety of settings. For $i=1,\dots, n$ and each $f\in \cF$, let $X_{f,i}$ be i.i.d.~copies of $X_f$. By (\ref{eq:annealedpower}) in Proposition~\ref{prop:useful-properties-esi} combined with Proposition~\ref{prop:esi_sums} we automatically have that, for all $\eta > 0$, for all $ f \in \cF$, $\sum_{i=1}^n X_{f,i} - n \A^{\eta}[X_f] \stochleq_\eta 0$ for every ESI function $\eta$. 
Zhang's bound, which using ESI notation we can give simultaneously in its expectation and in-probability version, is quite simply the result of applying the  PAC-Bayes bound of Part 3 in Proposition~\ref{prop:pacbayes} to these ESIs, and then dividing everything by $n$: 
\begin{equation}\label{eq:zhang}
\E_{\bar{f} \sim \hat{\Pi}} \left[\frac{1}{n} \sum_{i=1}^n X_{\bar{f},i} - \A^{\eta}_{Z\sim P}[X_{\bar{f}}] \right] \stochleq_{n \eta} \frac{\KL(\hat{\Pi} \| \Pi_0)}{n \eta},
\end{equation}
where we note that, as defined in the introduction, in Zhang's work the $X_f$ represent minus excess risks, $X_f = L_f$ with $L_f = L_f(Z) = \ell_f(Z) - \ell_{f^*}(Z)$.
The basic bound can then further be refined by bounding $\A^{\eta}$. In the setting of well-specified density estimation (in which $f^*$ is the density of the underlying $P$) the $f$'s represent densities and $\ell_f(z) = - \log f(z)$ is the log-score. For fixed $\eta=1/2$ $\A^{1/2}$ is the R\'enyi divergence of order $1/2$ \Citep{ErvenH14}, which is an upper bound on the Hellinger distance. In that case, Zhang's bound becomes a risk bound for density estimation. For other loss functions we proceed as follows: since the bound holds under no further conditions at all, for every $\eta > 0$, it still holds if we replace $\eta$ by an arbitrary ESI $u$. $\A^{u}_{Z\sim P}[X_f]$ can then be bounded in terms of $\E_{Z\sim P}[X_f]$ for appropriate $\gamma$-strong ESI function $u$. This is what was done in \cite[Lemma 21]{GrunwaldM20} which we restated here in ESI language as Proposition~\ref{lem:witness}---we essentially followed their reasoning in the introduction while avoiding the explicit use of $\A^{u}$ there. 
}
\end{example}

\begin{example}{\label{ex:begin}{\bf \ [B\'egin et al.'s unified derivation]} \rm 
\cite{begin2016pac} implicitly used the first result (Part 1 of the proposition above) to unify several PAC-Bayesian {\em generalization\/} bounds. They work in the same statistical learning setup as in the introduction, so $\ell_f(Z)$ represents the loss predictor $f$ makes on outcome $Z$, and the aim is to bound, with high probability, the expected loss $\E_{\bar{f} \sim \hat{\Pi}} \E_{Z \sim P}[\ell_{\bar{f}}(Z)]$ of the learned distribution on classifiers $\hat{\Pi}$, when applied by drawing a $\bar{f}$ randomly from $\hat{\Pi}$, in terms of the behaviour of $\hat{\Pi}$ on the training sample, $\E_{\bar{f} \sim \hat\Pi}
\left[\frac{1}{n} \sum_{i=1}^n \ell_{\bar{f}}(Z_i) \right]$.
In our (significantly compressed) language, they reason as follows:  
let $a \in \reals^+ \cup \{\infty\}$ and suppose we have a jointly convex divergence $\Delta: [0,a] \times [0,a] \rightarrow \reals^+_0$, where by “divergence” we mean that  $\Delta(c,c') \geq 0$ for all $c,c'\in [0,a]^2$ and $\Delta(c,c') = 0$ iff $c = c'$.
Upon defining $X_f = \Delta(n^{-1} \sum_{i=1}^n \ell_f(Z_i), \E_P[\ell_f])$, 
we get, using Jensen's inequality,
\begin{align*}
&    \Delta
\left(\E_{f \sim \hat\Pi}
\left[\frac{1}{n} \sum_{i=1}^n \ell_f(Z_i) \right], \E_{f \sim \hat\Pi} \E_{Z \sim P} [ \ell_f(Z)]
\right)\\ \leq 
    & \E_{f \sim \hat\Pi} \left[ 
    \Delta
    \left(\frac{1}{n} \sum_{i=1}^n \ell_f(Z_i),\E_{Z \sim P} [ \ell_f(Z)] \right)
    \right] \\ 
    & =  \frac{1}{n}  \E_{f \sim \hat\Pi} \left[ n \cdot
    \Delta
    \left(\frac{1}{n} \sum_{i=1}^n \ell_f(Z_i),\E_{Z \sim P} [ \ell_f(Z)] \right)
    \right]
    \\
&    \stochleq_{\eta} 
    \frac{1}{n} \left( \A^\eta_{(Z_i,f)\sim P \bigotimes \Pi_0} \left[ n \Delta \left(\frac{1}{n} \sum_{i=1}^n \ell_f(Z_i), \E_{Z\sim P}[\ell_f(Z)] \right)\right]  +  \frac{\KL(\hat{\Pi} \| \Pi_0)}{\eta}  \right),
\end{align*}
where $\A^\eta_{(Z_i, f)\sim P \bigotimes \Pi_0}[ \Delta(n^{-1} \sum_{i=1}^n \ell_f(Z_i), \E_P[\ell_f])]$ can be further bounded to get some well-known existing PAC-Bayes bounds such as the {\em Langford-Seeger-Maurer bound\/} \citep{Alquier21}. The latter is obtained by taking $\Delta$ as the KL divergence and letting $\eta$ depend on $n^{-1} \sum \ell_f(Z_i)$ in a clever way.  
}
\end{example}

\commentout{
\begin{corollary}
  Let $\bracks{X_f}_{f\in\calF}$ be a family of random variables. Suppose that
  there exist constants $C,\eta>0$ such that for each $f\in\calF$,
  $X_f \stochleq_{\eta^*} C$ and let $\eta > 0$. Then for any two equivalent distributions $\Pi$
  and $\hat{\Pi}$ on $\calF$
  \begin{equation}\label{eq:firstpacbayes}
    \E_{\hat{\Pi}} [X_f] \stochleq_{\eta} \frac{\KL(\hat{\Pi} \| \Pi_0)}{\eta} +
    \frac{1}{2}\frac{v\eta}{1-c\eta}
  \end{equation}
  for $0\leq \eta < 1/c $ and some constants $c,v>0$.
\end{corollary}

{\em Question: can we do this with variable $\eta$ depending on data, perhaps
  incuring a $\log \log n$ cost for choice of $\eta$?}
}

\section{ESI with random $\eta$}\label{sec:random}

In some applications, we will want $\eta$ to be estimated itself in terms of
underlying data, i.e.~it becomes a random variable $\hat\eta$; trying to learn $\eta$ from the data is a recurring theme in one of the author's work, starting with his first learning theory article \citep{grunwald1999viewing}, and shown to be possible in some situations using the {\em safe-Bayesian algorithm\/} \citep{grunwald_safe_2012} while leading to gross problems in others \Citep{GrunwaldO17}. Also, the fine-tuning of parameters in several PAC-Bayes bounds (e.g.~Catoni's [\citeyear{catoni_pac-bayesian_2007}] or the one in \cite{mhammedi_pac-bayes_2019}) can be reinterpreted in terms of an $\eta$ determined by the data. The goal of the present section is to extend
the  ESI definition to this case, allowing us to get a more general idea of what is possible with random $\eta$ than in the specific cases treated in the aforementioned articles. In this section we only consider strong ESIs, i.e.~$X\stochleq_{\eta} Y$ rather than $X \stochleq_u Y$.

Interestingly, many properties still go through for ESI with random $\eta$, but the 
in-expectation implication gets weakened---and its proof is not trivial any
more. 
\begin{definition}[ESI with random $\eta$]
  Let $\hat\eta$ be a random variable with range $H \subset \reals^+$ such
  that $\inf H > 0$. Let $\{X_{\eta} : \eta \in H \}$ and
  $\{Y_{\eta} : \eta \in H \}$ be two collections of random variables. We
will write   
  \begin{align}\label{eq:randomesi}
    X_{\hat\eta} \stochleq_{\hat\eta} \ \ Y_{\hat\eta}  \ \ \text{ as shorthand for }
    \ \
    \hat\eta(X_{\hat\eta} - Y_{\hat\eta})\stochleq_1 0
\end{align}
\end{definition}

We can still get an in-expectation result from random-$\eta$-ESI with a small
correction. It is trivial to give bounds for the expectation with $1/\eta_{\min}$---with $\eta_{\min}$ the largest lower bound of $H$---
as a leading constant. However, since we want to work with $\hat\eta$ that
are very small in “unlucky” cases but large in lucky cases, and we want to
exploit lucky cases, this is not good enough. The following result, which extends Proposition~ \ref{prop:esi_characterization} and~\ref{prop:exponential_tail_esi} to the random $\eta$ case and, in contrast to those propositions, is far from trivial, shows that we
can instead get a dependence of the form $1/\hat{\eta}$, which is of the same
order as what we lose anyway, even for fixed $\eta$, if we want our results to hold with high
probability. 


We let $\{X_{\eta}:  \eta \in \cG \}$ and $\{Y_{\eta}:  \eta \in \cG \}$ be any two collections of random variables.
\begin{theorem}
\label{thm:ESIbothways}
  Let $\cG$, $X_{\eta}$ and $Y_{\eta}$ be as above, with $H$ finite. We have:
  \begin{enumerate}\item If $X_{\hat\eta} \stochleq_{\hat\eta} \ Y_{\hat\eta}$, then for any $\delta \in]0,1[$,
    \begin{align}\label{eq:juditha}
      &\prob \left( X_{\hat\eta} \leq Y_{\hat\eta} + \frac{\log \frac{1}{\delta}}{\hat\eta} \right) \geq 1-\delta, \\
     \text{and} \quad & \E \left[X_{\hat\eta}  \right] \leq \E \left[ Y_{\hat\eta} +\frac{1}{\hat\eta}\right]. \label{eq:inexp}
    \end{align}
\item As a partial converse, if (\ref{eq:juditha}) holds, then we have
    \begin{align}
    \label{eq:partialconverse}
      X_{\hat\eta} \stochleq_{\frac{\hat\eta}{2}}
      Y_{\hat\eta} + \frac{2\log 2}{\hat\eta}.
    \end{align}
  \end{enumerate}
\end{theorem}

\begin{remark} {\rm 
  The following simple example shows that even though
  $\E[\rme^{\hat\eta W_{\hat\eta}}] \leq 1$ it can happen that
  $\E[W_{\hat\eta}]$ is unbounded, showing that in general one cannot get rid of the additive $1 /\hat\eta$ on the right-hand side of \eqref{eq:inexp}:
  Let $H=\{\eta_1, \eta_2 \}$ and
  $W_{\eta_1} \equiv C_1 < 0$ and $W_{\eta_2} \equiv C_2 >0$. We then set
  $\eta_1=\frac{1}{-C_1}$ and $\eta_2=\frac{1}{C_2}$; note that $\eta_1, \eta_2 > 0$ as required. The term
  $\E[\rme^{\hat\eta W_{\hat\eta}}]$ does then not depend on $C_1$ and $C_2$
  and computes to
  $$\E[\rme^{\hat\eta W_{\hat\eta}}]=p(\eta_1) e^{-1}+p(\eta_2) e^{1}$$
  This term is smaller than $1$ if we set for example
  $p(\eta_1)=\frac{3}{4}$. But for $C_2 \rightarrow \infty$ we observe that
  $\E[W_{\hat\eta}] \rightarrow \infty$.
  }
\end{remark}
\subsection{Additional properties for random ESI: transitivity, PAC-Bayes on $\hat\eta$}
Having established that the basic interpretation of an ESI as simultaneously expressing inequality in expectation and in probability still holds for the random $\eta$ case, we may next ask whether the additional properties we showed for strong ESIs still hold in the random case, or even with random variables $X_{\eta}$ indexed by $\eta$ rather than $f$. We do this for the summation and transitivity properties of Section~\ref{sec:esisums} and the PAC-Bayesian results of Section~\ref{sec:pacbayes}.
\paragraph{Random $\eta$ ESI Sums and Transitivity}
In what follows given variables $Z_1,\dots,Z_n$, $n\in \mathbb{N}$, in some set $\cZ$, we denote $$Z^{n\setminus i} \coloneqq (Z_1,\dots,Z_{i-1},Z_{i+1},\dots, Z_n)\in \cZ^{n-1}.$$
The following is a result analogous to Proposition~\ref{prop:esi_sums}, Part 2, with “negative correlation” replaced by “ESI holding conditionally given all variables except 1”. Of course, it would be interesting to extend both results to make them more similar; whether this can be done will be left for future work.
\begin{proposition}[{\bf ESI for sums and Transitivity with Random $\eta$}]
\label{prop:trans} Let $Z_1,\dots, Z_n\in \cZ$ be i.i.d random variables distributed according to $P$, and let $\cG$ be a finite subset of $\reals^+$. For every $\eta\in \cG$ and $i\in[n]$, let $X_{i,\eta}: \cZ^n \rightarrow \reals$ be a measurable function such that 
\begin{align} 
\text{for all } i\in[n]\text{ and } z^{n\setminus i} \in \cZ^{n-1}, \quad X_{i,\eta}(z_1,\dots,z_{i-1},Z,z_{i+1},\dots,z_n)  \stochleq_{\eta} 0. \label{eq:assum}\end{align}
Then, for any random $\hat\eta \in \cG$, we have
\begin{align}
\E \left[\sum_{i=1}^n X_{i,\hat \eta}(Z^n) \right] \leq \E \left[\frac{\log |\cG|+1}{\hat\eta} \right].
\end{align}
\end{proposition}

\paragraph{Random $\eta$ and PAC-Bayes}
We now investigate whether strong ESIs for individual fixed $\eta$'s  are as easily combined into an ESI involving all $\eta$'s, the particular $\eta$ chosen in a data-dependent manner, as they are for individual $X_f$'s. There we used general PAC-Bayesian combinations with arbitrary `posterior' (data-dependent) $\hat{\Pi}$ on $f \in \cF$. Here we  consider the analogue with a data-dependent distribution $\hat{\Pi}$ on $\hat\eta$. We find that the resulting bound is slightly different, involving the likelihood ratio between posterior and prior for the chosen $\hat\eta \sim \hat{\Pi}$ rather than in expectation over $\hat{\Pi}$ (which would be the direct analogue of the PAC-Bayesian result Proposition~\ref{prop:pacbayes}) 
 Still, if  we focus on the special but important case with $\hat{\Pi}$ a degenerate distribution, almost surely putting all its mass on a single estimator $\hat\eta$, then we get a precise analogy to the PAC-Bayes result. 

\begin{proposition}{\bf [PAC-Bayes on Random $\eta$]}
\label{prop:randeta} Let $\Pi_0$ be any prior distribution on $\cG$, and $\hat{\Pi}:\Omega \rightarrow \Delta(\cG)$ be any random estimator such that $\hat{\Pi}(\omega)$ is absolutely continuous with respect to $\Pi_0$, for all $\omega \in \Omega$. If $X_{\eta}\stochleq_{\eta} 0$, for all $\eta \in \cG$, then for $\hat\eta \sim \hat\Pi$, we have:
\begin{align}
 X_{\hat\eta}  \stochleq_{\hat\eta}  \frac{\left. \log  \frac{\rmd\hat\Pi}{\rmd\Pi_0}\right\rvert_{\hat\eta} }{\hat\eta}.
\end{align}
\end{proposition}
\begin{proof}
For $\eta \in \cG$, let $W_{\eta}$ be the random variable defined by $W_{\eta} \coloneqq  X_{\eta} - \frac{1}{\eta}\left.\log\frac{\rmd \hat\Pi}{\rmd \Pi_0}\right\rvert_{\eta}.$ We have
\begin{align}
\E \left[
    e^{\hat\eta W_{\hat\eta}}\right] = \E_{Z\sim P} \E_{\hat\eta \sim \hat\Pi}\left[
    e^{\hat\eta W_{\hat\eta}}\right] 
       =  \E_{Z\sim P} \E_{\hat\eta \sim \hat \Pi} \left[\rme^{\eta X_{\hat\eta} - \log \left. \frac{\rmd \hat\Pi}{\rmd \Pi_0}\right|_{\eta}} \right]
      = \E_{Z\sim P} \E_{\eta \sim  \Pi_0} \left[\rme^{\eta X_{\eta}} \right]  \leq 1,
\end{align}
where the last step follows from the fact that $X_{\eta}\stochleq_{\eta} 0$, for all $\eta \in \cG$. This completes the proof.
\end{proof}
\begin{corollary}
\label{prop:randetacor} Let $\Pi_0$ be any prior distribution on $\cG$, and $\hat{\Pi}:\Omega \rightarrow \Delta(\cG)$ be any random estimator such that $\hat{\Pi}(\omega)$ is absolutely continuous with respect to $\Pi_0$, for all $\omega \in \Omega$. If $X_{\eta}\stochleq_{\eta} 0$, for all $\eta \in \cG$,  then for any $0<\delta 1$ and $\hat\eta \sim \hat\Pi$:
\begin{align}
P \left(X_{\hat\eta}   \leq   \frac{\left. \log \frac{\rmd\hat\Pi}{\rmd\Pi_0}\right\rvert_{\hat\eta} +\log \frac{1}{\delta}}{\hat\eta} \right) &\geq 1-\delta, \\  \text{and}\quad     \E_{Z\sim P}\left[ X_{\hat\eta}  - \frac{\left. \log \frac{\rmd\hat\Pi}{\rmd\Pi_0}\right\rvert_{\hat\eta} +1}{\hat\eta}\right] &\leq 0.
\end{align}
In particular, if $\hat{\Pi}$ a.s. puts mass 1 on a particular $\hat\eta$, where $\hat\eta$ is a random variable taking values in $\cG$, and $\cG$ is a countable set, $\Pi_0$ having probability mass function $\pi_0$, then the $ \left. \log \frac{\rmd\hat\Pi}{\rmd\Pi_0}\right\rvert_{\hat\eta}$ term is equal to $- \log \pi_0(\hat\eta)$.
\end{corollary}
\begin{proof}
The result follows by applying Propositions \ref{prop:randeta} and Theorem~\ref{thm:ESIbothways}  to the random variable $W_{\eta} \coloneqq  X_{\eta} - \frac{1}{\eta}\left.\log \frac{\rmd \hat\Pi}{\rmd \Pi_0}\right\rvert_{\eta}$, $\eta \in \cG$.
\end{proof}

\commentout{
\begin{proposition}
  Let $\bracks{X_\eta}_{\eta > 0}$ be a family of random variables indexed by a positive real parameter $\eta$. Let
  $\hat\eta$ be a random variable taking values on $\reals^+$ and assume that
  $X_{\hat\eta}\stochleq_{\hat\eta} 0$. Then for any $t>0$ and $x>0$ it
  holds that
  \begin{equation*}
    P(X_{\hat\eta} \geq x) \leq P(1/\hat\eta \geq t) + \rme^{-x/t}
  \end{equation*}
\end{proposition}

\begin{proof}
We proceed to bound
\begin{align*}
  P(X_{\hat\eta} \geq x)
  &=
    P(X_{\hat\eta} \geq x, 1/\hat\eta \geq t)
    +
    P(X_{\hat\eta} \geq x, 1/\hat\eta  < t) \\
  &\leq
    P(1/\hat\eta \geq t)
    +
    P(\hat\eta X_{\hat\eta} \geq \tfrac{x}{t},\ 1/\hat\eta < t)\\
  &\leq
    P(1/\hat\eta \geq t)
    +
    P(\hat\eta X_{\hat\eta} \geq x/t)\\
  &\leq
    P(1/\hat\eta \geq t)
    +
    \rme^{-x/t}
\end{align*}
which is what we were after.
\end{proof}
\begin{corollary}
  Under the same assumptions and the additional assumption that $\E[1/\hat\eta]<\infty$
  \begin{equation*}
    P(X_{\hat\eta}\geq x) \leq \frac{\E[1/\hat\eta]}{x}\paren{1 + \log\paren{\frac{x}{\E[1/\hat\eta]}}}
  \end{equation*}
\end{corollary}
\begin{proof}
  Use Markov's inequality and take $t = x/\log\frac{x}{\E[1/\hat\eta]}$
\end{proof}
}

\commentout{
We can use the following:
\begin{proposition}
  Let $H$ be a finite subset of $\reals^+$ and let $\{U_{\eta} : \eta \in H \}$
  be a finite collection of RVs. Suppose that for all $\eta \in H$,
  $U_{\eta} \stochleq_{\eta} 0$. Then for arbitrary estimators $\hat{\eta}$
  with support $H$, we have that
  $$
  U_{\hat\eta} \stochleq_{\hat\eta} \frac{\log |H|}{\hat\eta}
  \text{\ \ and hence \ \ } V_{\hat\eta} \stochleq_{\hat\eta} 0
$$
for $V_{\eta} := U_{\eta} - (\log |H|)/\eta$.
\end{proposition}
\begin{proof}
$  \E[\rme^{\hat\eta(U_{\hat\eta}-\frac{\log |H|}{\hat \eta})}] = \E[\rme^{\hat\eta U_{\hat\eta}} \frac{1}{|H|}]\leq \sum\limits_{\eta \in H} \E[\rme^{\eta U_{\eta}}]\frac{1}{|H|} \leq \frac{|H|}{|H|}=1$
\end{proof}
}


\section{Non-iid Sequences}
\label{sec:non-iid-sequences}
Here we extend  our previous results to sequences of random variables $X_1, X_2, \ldots$ that might not be
independent and identically distributed. 
We find that, if an ESI hold for each $X_i$ conditionally on the past, ESI statements about the sums of the $X_i$'s remain valid under optional stopping, thereby connecting ESIs to the recent surge of work in {\em anytime-valid confidence sequences}, {\em e-values, e-variables\/} and {\em e-processes\/} \citep{GrunwaldHK23,ramdas2023savi}. As a consequence, we reprove Wald's identity, a well-known result in sequential analysis dating back to the 1950s, and show that it is related to Zhang's inequality treated before, and implies that Zhang's inequality remains valid under optional stopping. Relatedly, it has recently been noted that PAC-Bayesian inequalities are closely related to e-processes as well \citep{jang23tighter,chugg2023unified}.
Let us clarify the straightforward connection between e-variables as defined in the above references and strong ESIs. Formally, e-variables $S$ are  defined relative to some random variable $Y$ and a {\em null hypothesis\/} $\calH_0$, a set of distributions on $Y$. We call  nonnegative random variable $S$ an e-variable relative to $Y$ and $\calH_0$ if it can be written as a function $S=S(Y)$ of  $Y$ and, for all $P \in \calH_0$, $\E_P[S(Y)]\leq 1$.   To clarify the connection to ESIs,  let $\{X_f : f \in \cF\}$ be a family of random variables with $P_f$ the marginal distribution of $X_f$ as induced by $P$, and suppose that  $\{X_f: f \in \cF \}$ all satisfy $X_f \stochleq_{\eta^*} 0$. Suppose that we observe random variable $Y$. Under the null hypothesis, $Y= X_f$ for some $f \in \cF$ (they take on the same values). Equivalently, under the null hypothesis, $Y \sim P_f$ for some $f \in \cF$, i.e.~$\cH_0 = \{P_f: f \in \cF \}$. Then, clearly,  $S(Y) := \exp(\eta^* Y)$ is an e-variable. 
We will not further exploit or dwell on this fact below, but rather concentrate on the development of ESI for random processes. 
\begin{definition}[Conditional ESI]
  Let let $X$ and $Y$ be two random variables defined on the same probability
  space $(\Omega, \calF, P)$ and let $\calG\subseteq \calF$ be a
  $\sigma$-algebra. Define
  \begin{equation*}
    X \stochleq_{\eta,\calG} Y \text{  if and only if  }
    \A^\eta[X - Y|\calG]\leq 0 \text{ almost surely,}
  \end{equation*}
  where we call
  $\A^\eta[X-Y|\calG] = \frac{1}{\eta}\log \E[\rme^{\eta(X - Y)}|\calG]$ the
  conditional annealed expectation of $X-Y$ given $\calG$.
\end{definition}

The following properties can be checked; they follow from the standard properties of the conditional expectation---“pulling out known factors”, and the tower property. 
\begin{proposition}
  Let let $X$ be an $\calF$-measurable random variable and let
  $\calH\subseteq\calG\subseteq\calF$ be $\sigma$-algebras. The following hold:
  \begin{enumerate}
  \item \label{item:conditional-esi-measurable}If $X\stochleq_{\eta,\calG}0$ and $X$ is
    $\calG$-measurable, then $X\leq 0$ almost surely.
  \item \label{item:conditional-esi-independent}If $X\stochleq_{\eta,\calG} 0$,
    then $X \stochleq_{\eta} 0$.
  \item \label{item:conditional-esi-tower} If $X\stochleq_{\eta,\calG}0$, then
    $X\stochleq_{\eta,\calH}0$.
  \end{enumerate}
\end{proposition}
\begin{proof}
  \ref{item:conditional-esi-measurable} follows from the fact that if $X$ is
  $\calG$-measurable, then $\A^\eta[X|\calG] = X$.
  \ref{item:conditional-esi-independent} follows from the fact that
  $\A^\eta[X|\calG] \leq 0$ implies that
  $\A^\eta[X]= \A^\eta[\A^\eta[X|\calG]]\leq 0$.
  \ref{item:conditional-esi-tower} follows from the tower property of
  conditional expectations because
  $ \A^\eta[X|\calH]= \A^\eta[\A^\eta[X|\calG]|\calH] \leq 0$.
\end{proof}

Let $(\Omega, \mathbbm{F} = (\calF_t)_{t\in\mathbbm{N}}, P)$ be a filtered
probability space. Let $(X_t)_{t\in\mathbb{N}}$ be a sequence of random
variables adapted to $\mathbbm{F}$ and assume that $X_t\stochleq_{\eta,t-1}0$
(where we write $X_t\stochleq_{\eta,t-1}0$ instead of
$X_t\stochleq_{\eta,\calF_{t-1}}0$ to avoid double subindexes). This statement
expresses the fact that  $(\prod_{s\leq t}\rme^{\eta X_s})_{t\in\mathbbm{N}}$ is a supermartingale.

\begin{proposition}\label{prop:adaptation}
  Let $(X_t)_{t\in\mathbbm{N}}$ be an adapted sequence such that
  $X_t\stochleq_{\eta,t-1}0$ for each $t$ and for some $\eta>0$. Let $\tau$ be an
  almost surely bounded stopping time with respect to
  $(X_t)_{t\in\mathbbm{N}}$. Then, if $S_t = \sum_{s\leq t} X_s,$
  \begin{equation*}
    S_{\tau} \stochleq_\eta 0.
  \end{equation*}
\end{proposition}
\begin{proof}
  The result is an application of the Optional Stopping Theorem.
\end{proof}
We now present two applications of this result, Example~\ref{ex:waldmeetszhang} and Proposition~\ref{prop:esimax}.
\begin{example}\label{ex:waldmeetszhang}{\ \bf [Zhang meets Wald]}
{\rm 
Let $X_1, X_2, \ldots$ be i.i.d.~copies of some random variable $X$, 
fix arbitrary $\eta > 0$ and let $Z_i = X_i - \A^{\eta}[X]$. Then the $Z_i$ are also i.i.d., hence $(Z_t)_{t \in {\mathbb N}}$ is adapted, and by (\ref{eq:annealedpower} in Proposition~\ref{prop:useful-properties-esi}, they satisfy $Z_t \stochleq_{\eta,t-1}0$. Therefore  we can use Proposition~\ref{prop:adaptation} to infer that for any a.s. bounded stopping time $\tau$, with  $S_t := \sum_{i=1}^ t Z_i$  that $S_{\tau} \stochleq_{\eta} 0$, 
i.e.
\begin{equation}\label{eq:waldidentity}
\sum_{i=1}^{\tau} X_i  - \tau \cdot \A^{\eta}[X] \stochleq_{\eta} 0,
\end{equation}
which must hold for all $\eta > 0$ and thus also if $\eta$ is replaced by any ESI function $u$. 
But (\ref{eq:waldidentity}) is just the celebrated {\em Wald identity\/} \citep{Skorokhod91} as expressed in ESI notation, which we have thus reproved. (the Wald identity is not to be confused with the more well-known basic Wald's equation, which says that $\E[S_{\tau}] = \E[\tau] \cdot \E [X]$).
We may now, just as in Example~\ref{ex:zhang}, combine this with a PAC-Bayes bound and then divide everything by $\tau$ to get, for a family of random variables $\{X_f: f \in \cF\}$ with $X_{f,1}, X_{f,2}, \ldots$ i.i.d.~copies of $X_f$ as in the introduction,
\begin{equation*}
   \E_{\bar{f} \sim \hat\Pi} \left[\frac{1}{\tau} \sum_{i=1}^{\tau} X_{\bar{f},i}  -   \A^{\eta}[X_{\bar{f}}] \right]\stochleq_{\tau \eta} \frac{\KL(\hat{\Pi} \| \Pi_0)}{\tau \eta}. 
\end{equation*}
We see that this is identical to  {\em Zhang's inequality\/} (\ref{eq:zhang}), which we have therefore shown to be ``anytime valid'' (it holds for any stopping time $\tau$), something that, it seems, has not been noted before. 
Since the scaling $\tau\eta$ is now data-dependent, we have to use Theorem~\ref{thm:ESIbothways} rather than Proposition~\ref{prop:esi_characterization} if we want to turn this in an in-probability or in-expectation bound though. 
}    
\end{example}

Finally, we note that using Ville's maximal inequality\footnote{This inequality is also commonly
  attributed to J.L. Doob.} \citep[p.35]{ville_etude_1939} we can obtain the
following proposition. 
\begin{proposition}\label{prop:esimax}
  Let $(X_t)_{t\in\mathbbm{N}}$ be a sequence of random variables such that
  $X_t\stochleq_{\eta^*,t-1} 0$ for each $t$ and some $\eta^*>0$. Let
  $0<\eta<\eta^*$. Then, there is a fixed constant $c$ such that
  \begin{equation*}
    \sup_{t\in\mathbbm{N}}X_t\stochleq_\eta c.
  \end{equation*}
\end{proposition}

\begin{proof}
  By Ville's maximal inequality,
  \begin{align*}
    P\paren{\sup_{t\in\mathbbm{N}} X_t\geq x}
    &=
      P\paren{\sup_{t\in\mathbbm{N}} \rme^{\eta^* X_t}\geq \rme^{\eta^* x}} \leq
      \E[\rme^{\eta^* X_1}]\rme^{-\eta^* x} \leq \rme^{-\eta^* x}.
  \end{align*}
  The result follows from Proposition \ref{prop:exponential_tail_esi}.
\end{proof}

\section{Discussion}

It is sometimes the case that half  the way to solving a problem is finding the correct notation to state it. We have emphasized that many results in probability theory and statistical learning theory---especially PAC-Bayesian bounds---are obtained through bounds for cumulant generating functions. In this article we have have introduced a notational device with the goal of systematizing such bounds. The result is the Exponential Stochastic Inequality (ESI), which the authors have found helpful---we do not claim its absolute superiority, though. The strong ESI $X\stochleq_\eta 0$ can be thought of as an interpolation between positivity in expectation (the case that $\eta \downarrow 0$) and almost-sure positivity ($\eta\to \infty$). Its main properties, shown in Section~\ref{sec:basic}, allow for the derivation of high-probability and in-expectation bounds, and its transitivity-like property allows for chaining such bounds in a way that is superior to a straightforward union bound. 

Inventing new notation is, however, a contentious affair. We have found the community to be rather conservative about notational changes. Like many things, this has two sides. On the positive side, it allows for easy understanding of a wide variety of articles at a low overhead. Standard notation serves as a \textit{lingua franca} for conveying mathematical ideas. On the other side, sometimes good ideas are obscured for the sole reason that they are awkward to write in standard notation. We believe that in these cases---such as, we argue, PAC-Bayesian bounds---, new notation can help clarify and systematize the key techniques of the field. This is not a new idea; for instance, mathematicians (sometimes) and physicists (more often) have been inventing new notation for the better part of last century---think of Feynman diagrams or Einstein's summation convention. Hopefully, as it has already happened in other areas, new notation will help easier communication, provide a deeper understanding of the present techniques, and help the rise of new ones. 

Having said that, we note once more that the ESI does have limitations. Of course, not all tails are exponential, and not all bounds are obtained through the analysis of cumulant generating functions. The arguments that we have presented using the ESI are consequences of the use of  a particular convex duality relation---the one that exists between the cumulant generating function and the Kulback-Leibler divergence. A particular and interesting extension of the ESI might come from applying the same reasoning to other convex duality relationships. For example, \citet{LugosiN22} provide PAC-Bayesian-like bounds based on other convex dualities; their work might be considered a first step in this direction. This enterprise is worthwhile because convex duality arguments are the bread and butter of statistical learning theory, online learning and optimization.

\section{Acknowledgements}
This manuscript benefited enormously from several conversations with Wouter Koolen, Tim van Erven, Nishant Mehta and Odalric-Ambrym Maillard; in particular, some results in  Maillard's (\citeyear{maillard_mathematics_2019}) habilitation, although not directly used,  were inspirational to the developments in Section~\ref{sec:random}.  This research was supported by  the Dutch Research Council (NWO) via research project 617.001.651, {\em Safe Bayesian Learning}.
  \bibliography{local,master,muriel_references,MDL}
  \appendix
\section{Proofs for Section~\ref{sec:basic}}
\begin{proof}{\bf (of Proposition~\ref{prop:exponential_tail_esi})}
Since $\exp(\eta Z) \leq \exp(\eta Z_+)$, we have for each $0 < \eta<b$, 
  using also Fubini's theorem and the tail condition on $Z$
  \begin{align*}
&     \E[\rme^{\eta Z} - 1]  \leq \E[\rme^{\eta Z_+} - 1]
    = \E[\int_0^{Z_+}\eta\rme^{\eta Z}\rmd z]
    =\eta\int_0^\infty P(Z\geq z )\rme^{\eta z}\rmd z
    \leq \\ & a\eta\int_0^\infty \rme^{-(b - \eta)z}\rmd z
    = \frac{a \eta }{b-\eta},
  \end{align*}
  which means that
  \begin{equation}\label{eq:subgamma_bound}
    Z\stochleq_{\eta}\frac{1}{\eta}\log\paren{1+\frac{a\eta}{b-\eta}}
  \end{equation}
  For the first claim, pick $0\leq \eta^*< b$ and call $c$ the right hand side
  of \eqref{eq:subgamma_bound} when evaluated at $\eta =\eta^*$. The result
  follows by item \ref{item:esi-properties-convexity} in Proposition
  \ref{prop:useful-properties-esi}, ahead. The converse follows from
  \begin{equation*}
    P(X\geq Y + \epsilon) \leq \rme^{\eta(\A^\eta[X-Y] - \epsilon)}\leq\rme^{\eta(c-\epsilon)}
  \end{equation*}
  with $a = \rme^{\eta c}$, and $b = \eta$.
\end{proof}  

\begin{proof}{\bf (of Proposition~\ref{prop:markov})}
  We can write
  \begin{equation*}
    \E[X]= \E[X \ | \ X\geq 0]P(X\geq 0) + \E[X \ | \ X<0]P(X<0).
  \end{equation*}
  We will bound both terms on the right hand side from bellow. For the first
  one, use Markov's inequality and the definition of conditional expectation to
  obtain that
  \begin{equation}
    \E[X \ | \ X\geq 0]P(X\geq 0) = \E[[X]_+] \geq a\prob(X\geq a). \label{eq:b1}
  \end{equation}
  For the second term, use the conditional version of Jensen's inequality to
  obtain that
  \begin{align}
    \E[X \ |\  X<0]
    &\geq -\frac{1}{\eta} \log \E[ e^{-\eta X} \ | \ X<0],\nonumber  \\
    &\geq -\frac{1}{\eta} \log \frac{1}{\prob(X<0)}. \label{eq:b2}
  \end{align}
  where the last inequality holds because by the assumption that
  $0\stochleq_\eta X$, which implies
  \begin{align*}
  1 \geq   \E[\rme^{-\eta X}] & = \E[\rme^{-\eta X}\ | \ X \geq 0]\prob(X\geq 0) + \E[\rme^{-\eta X} \ |
    \ X< 0]\prob(X < 0), \\
    &  \geq \E[\rme^{-\eta X} \ | \ X< 0]\prob(X < 0).
  \end{align*}
  Gathering \eqref{eq:b1} and \eqref{eq:b2} together implies
  \begin{equation*}
    \E[X] \geq aP(X\geq a)   - \frac{1}{\eta} \prob(X<0)\log
    \frac{1}{\prob(X<0)},
  \end{equation*}
  which after rearrangement implies the first inequality. The second inequality
  follows from maximizing the function $x\mapsto x\log(1/ x)$, which is a concave
  function that attains its maximum value $1/\rme$ at $x^* =1/\rme$.
\end{proof}
 
  \section{Proofs for Section~\ref{sec:weakesi}}
  \label{app:weakesi}
  Proposition~\ref{prop:newversionsubgammaapp} below strictly strengthens Proposition~\ref{prop:newversionsubgammamain}.
To see how, note that $(1) \Rightarrow (2)$ in Proposition~\ref{prop:newversionsubgammamain}  is implied by 1. below, noting that in Proposition~\ref{prop:newversionsubgammamain} we assume that $\{X_f\}_{f \in \cF}$ is regular, which implies the condition of (1). $(2) \Rightarrow (3)$ in Proposition~\ref{prop:newversionsubgammamain}  is implied by 2. below, again since in Proposition~\ref{prop:newversionsubgammamain} we assume that $\{X_f\}_{f \in \cF}$ is regular, together with the fact that $(2)$ in Proposition~\ref{prop:newversionsubgammamain} already implies that for all $f \in \cF$, the $X_f$ are subcentered. 
 $(3) \Rightarrow (4)$ in Proposition~\ref{prop:newversionsubgammamain}  is directly implied by 3. below and again the  fact that $(3)$ in Proposition~\ref{prop:newversionsubgammamain} already implies that for all $f \in \cF$, the $X_f$ are subcentered, so that $X_f \leq X-f - \E[X_f]$ for all $f \in \cF$. 
 $(4) \Rightarrow (1)$ in Proposition~\ref{prop:newversionsubgammamain}  is directly implied by 4. below. $(3) \Rightarrow (5)$ is implied by 5. below and $(5) \Rightarrow (3)$ is implied by $6.$ below.
    \begin{proposition}\label{prop:newversionsubgammaapp}
\begin{enumerate}\item 
Let $\{X_f\}_{f \in \cF}$ be a family of random variables
such that 
$\inf_{f \in \cF} \E[X_f] > - \infty$. Suppose there is an ESI function $u$ such that for all $f \in \cF$, $X_f \stochleq_{u} 0$. Then there are constants $C^*> 0$ and  $\eta^* >0$ such that uniformly for all $f \in \cF$, 
$X_f \leq  X_f - \E[X_f] \stochleq_{\eta^*} C^*$
(in particular, for all $f \in \cF$, $X_f$ is subcentered, i.e.~$\E[X_f] \leq 0$).
\item Suppose  $\sup_{f \in \cF} \Var[X_f] < \infty$. Suppose there is a constant $C^*> 0$ and a constant $\eta^* >0$ such that uniformly for all $f \in \cF$, 
$X_f - \E[X_f] \stochleq_{\eta^*} C^*$.
Then there exists $c,v > 0$ such that for all  $f\in\calF$, the $X_f$ are $(c,v)$-subgamma on the right, i.e.~they satisfy  (\ref{eq:sub_gamma_first_description}). 
\item 
Suppose there exists $c,v > 0$ such that for all $f\in\calF$, the $X_f$ are $(c,v)$-subgamma on the right. Then 
there is an ESI function $h$ such that for all $f \in \cF$, we have $X_f - \E[X_f]  \stochleq_{h} 0$ where $h$ is of the form 
$h(\epsilon) = C \epsilon \wedge \eta^*
$ 
for some constants $C > 0, \eta^* > 0$.
\item Suppose 
that for all $f \in \cF$, we have $X_f \leq X_f - \E[X_f]  \stochleq_{h} 0$ where $h(\epsilon) = C \epsilon \wedge \eta^*
$  for some $C, \eta^* > 0$. Then 
there is an ESI function $u$ such that for all $f \in \cF$, $X_f \stochleq_{u} 0$.
\item Suppose there exists $c,v > 0$ such that for all $f\in\calF$, the $X_f$ are $(c,v)$-subgamma on the right. Then for all $0 < \delta \leq 1$,  all $f \in \cF$,  (\ref{eq:boucheronsubgamma}) holds with probability at least $1-\delta$.
\item Suppose there exists $c,v > 0$ such that for all $f\in\calF$, the $X_f$ are subcentered and, for each $f \in \cF$, for each $0 < \delta \leq $1, with probability at least $1- \delta$, (\ref{eq:boucheronsubgamma}) holds. Then: \\ there exists $a > 0$ and  a differentiable function $h: \reals^+_0 \rightarrow \reals^+_0$
with $h(\epsilon) > 0$ and $h'(\epsilon) \geq 0$ for $\epsilon > 0$, such that for all $f \in \cF$,  the $X_f$ are subcentered and $P(X \geq \epsilon) \leq a \exp(- h(\epsilon))$. 
\item Suppose the condition above holds. Then there is $C^* >0, \eta^* > 0$ such that uniformly for all $f \in \cF$, $X_f \leq X_f - \E[X_f] \stochleq_{\eta^*} C^*$. 
\end{enumerate}
\end{proposition}
\begin{proof}
{\em Part 1}. Let $C':=  - \inf_{f \in \cF} \Exp[X_f] < \infty$. 
Take $C'' > 0$ such that $\eta^* := u(C'') >0$. 
Then  $X_f \stochleq_{\eta^*}  C''$ and
$X_f - \Exp[X_f] \stochleq_{\eta^*} C'' + C':= C^*$
Also $\Exp[X_f] \leq \epsilon$ for all $\epsilon > 0$ and hence $\Exp[X_f] \leq 0$, which implies subcenteredness. 

{\em Part 2.} see Theorem~\ref{thm:muriel} and its proof below.

{\em Part 3.} 
Let $U_f =X_f -\E[X_f]$. The assumption of right subgammaness implies that for all $0 < \eta \leq \eta^*$ with $\eta^* = 1/(2c)$ and $C' = 2 v$, we have
$$
\E[\rme^{\eta U_f}] \leq \exp \left(\eta^2 
\cdot \frac{v}{1- c \eta} 
\right) \leq   \exp \left(\eta^2 
\cdot \frac{v}{1- (1/2)} 
\right) =  \exp \left(\eta^2 
\cdot C'
\right).
$$
Now take 
$h(\epsilon) = \epsilon/C'$ if $\epsilon \leq C'/2c$ and $h(\epsilon) = 1/2c$ for $\epsilon > C'/2c$. Then the above display implies that $\E[U_f] \stochleq_{h} 0$, and $h$ is seen to be equal to the $h$ in the proposition statement. 

{\em Part 4.} The premise implies that $\E[X_f] \leq 0$ for all $f \in \cF$ and $X_f \stochleq_h 0$, so we can (trivially) take $u=h$.

{\em Part 5.} (\ref{eq:boucheronsubgamma}) be rewritten as: for every $t > 0$, 
\begin{equation}\label{eq:boucheronsubgammab}
P(X_f > \sqrt{2 vt} + ct) \leq \exp(-t),
\end{equation}
which is shown to be implied by $(c,v)$-subgammaness on the right in \cite[Section 2.4]{boucheron_concentration_2013}.

{\em Part 6.} \cite[Section 2.4]{boucheron_concentration_2013}.
shows that (\ref{eq:boucheronsubgammab}) is equivalent to $P(X_f > t) \leq \exp(- h(t))$ where $h(t) = (v/c^2) h_1(ct /v))$, with $h_1(u) = 1+u - \sqrt{1 + 2u}$; this is of the required form. 

{\em Part 7.} If $h(0) > 0$, then we can simply apply Proposition~\ref{prop:exponential_tail_esi} to get the desired result; if $h(0) = 0$, apply the proposition with $a$ set in the proposition to $2a$. 
\end{proof}
 \begin{proof}{[\bf of Theorem~\ref{thm:muriel}]}
   Suppose that $U - \E[U] \stochleq_{\eta^*} C$  holds and suppose without
   loss of generality that $U$ is centered. Let
   $M = \rme^{\eta^*\A^{\eta^*}[U]} = \rme^{\eta^*C} $, which is finite
   by assumption. We bound the moments of the right part of $U$.
   \begin{align*}
     \E[(U_f)_+^n] &= \E[\int_0^{(U_f)_+}nu^{n-1}\rmd u] \\
     &= \E[\int_0^\infty P(U_f\geq u)nu^{n-1}\rmd u] \\
     &\leq \rme^{\eta^* \A^{\eta^*}[U_f]}n\int_0^\infty \rme^{-\eta^*
       u}u^{n-1}\rmd u\\
     &= n!  \frac{M}{\eta^{*n}}
   \end{align*}
Now let $v = \Var(U) + \frac{2M}{\eta^{*2}}$ and
   $c = \frac{1}{\eta^*}$. This means that
   $\E[(U)_+^n]\leq n!\frac{v}{2}c^{n-2}$ for $n\geq 3$ and $\E[U^2]\leq v$
   uniformly over $\calF$. The rest of the proof follows that of \citet[Theorem
   2.10]{boucheron_concentration_2013}: note that
     $\rme^x \leq 1 + x + \frac{1}{2}x^2$ for $x\leq 0$ so that
   \begin{equation*}
     \rme^x \leq 1 + x + \frac{1}{2}x^2 + \sum_{n\geq 3}\frac{x_+^n}{n!}
   \end{equation*}
   for all $x$. With this in mind, we obtain a bound on the moment generating
   function of $U$:
   \begin{align*}
     \E[\rme^{\eta U}]
     &\leq  1 + \E[U]  + \frac{1}{2}\eta^2\Var(U) +
       \sum_{n\geq 3}\frac{\eta^n\E[(U)_+^n]}{n!}\\
     &\leq 1 + \frac{1}{2}v\eta^2\sum_{n\geq 2}(c\eta)^{n-2}\\
     &= 1 + \frac{1}{2}\frac{v\eta^2}{1-c\eta}
   \end{align*}
   for $0<c\eta<1$. Taking logarithms on both sides, using that $\log(1 + x)\leq
   x$ for $x\geq 0$ and rewriting leads to
   \begin{equation*}
     \A^\eta[U]\leq \frac{1}{2}\frac{v\eta}{1-c\eta},
   \end{equation*}
   which is exactly what we were after. 
 \end{proof} 

 \paragraph{Proof of the Claim in Example~\ref{ex:varianceneeded}} 
Let $\eta < 1$. Using 
$P(-1/\eta \leq U < -1) = \int_{- 1/\eta}^{-1} p(u) d u =  (1-\eta^{v-1})/(v-1)$ and $\E[U^2 \cdot \ind{-1/\eta < U < 0} ]
= \int_{- 1/\eta}^{-1} u^2 p(u) d u = (1- \eta^{\nu-3})/(\nu-3)$ we find:
\begin{align*}
    \E[\exp(\eta U)] & \geq \E[\exp(\eta U) \cdot \ind{-1/\eta < U < 0}] + \E[\exp(\eta U) \cdot \ind{U \geq 0}] \\
    & \geq \E[(1+\eta U + U^2 \eta^2/4 ) \cdot \ind{-1/\eta < U < 0} ] + \E[(1+\eta U) \cdot  \ind{U \geq 0}] \\
    & \geq P(U > -1 /\eta) + \eta \E[U] + (\eta^2\cdot \exp(-1) \cdot \E[U^2 \cdot \ind{-1/\eta < U < 0} ] \\
    & = P(-1/\eta \leq U < -1) + (1- P(U \leq -1)) + \eta \E[U] + \eta^2 \cdot \exp(-1) \cdot \E[U^2 \cdot \ind{-1/\eta < U < 0} ] \\
    & = \frac{1-\eta^{\nu-1}}{\nu-1} + (1- \frac{1}{\nu-1}) + \eta \E[U] + \eta^2 \cdot \exp(-1) \cdot \frac{
    \eta^{\nu-3}-1}{3- \nu}  \\ & = 
    1- \frac{1}{\nu-1} \eta^{\nu-1} + \frac{\exp(-1)}{(3 - \nu)} \cdot (\eta^{\nu-1} - \eta^{2}).
\end{align*}
Since for $5/2 < \nu < 3$, we have $\exp(-1)/(3- \nu) > 1/(\nu-1)$, we find that 
$(\E[\exp(\eta U)] -1 )/\eta^2 \rightarrow \infty$ as $\eta \downarrow 0$, showing that right subgamma-ness is violated. 

 \section{Proofs for Section~\ref{sec:strongesi}}
 \label{app:strongesi}
Below we first state Theorem~\ref{thm:strongesib} and then  Lemma~\ref{lem:strongESIregularity}. We then show how, taken together, these two results imply the result in the main text, Theorem~\ref{thm:strongesi}, as an almost direct corollary.
After that, we provide first the proof of Lemma~\ref{lem:strongESIregularity} and then the proof of Theorem~\ref{thm:strongesib} itself, followed by the statement and proof of Lemma~\ref{lem:taylor}, a slight extension of the standard second-order Taylor approximation of the moment generating function that is crucial for proving Lemma~\ref{lem:strongESIregularity}.
\begin{theorem}\label{thm:strongesib}
\begin{enumerate}
\item Suppose 
$\{ - X_f: f \in \cF \}$ satisfies the $[0,b)$-Bernstein condition for some $0 < b \leq 1$ and the conclusion {\bf C1} of Lemma~\ref{lem:strongESIregularity}, Part 1 holds.  Then 
for all $\beta \in [0,b)$, all $c \geq 0$, all $0 < c^* < 1$, there exists  $\eta^{\circ} >0 $ and $C^{\circ} > 0$ such that for all $f \in \cF$, all $0 < \eta \leq \eta^*$,
$$X_f  + c \eta  X_f^2
- c^* \E[X_f] \stochleq_{\eta} C^{\circ} \eta^{1/(1-\beta)}.$$ 
\item Suppose there exists $\eta^{\circ} >0, C^{\circ} > 0$ such that, for some $0 < \beta \leq 1$, for all $f \in \cF$, $X_f \stochleq_{\eta^{\circ}} C^{\circ} \eta^{1/(1-\beta)}$ and the conclusion {\bf C2} of Lemma~\ref{lem:strongESIregularity}, Part 2 holds.  Then, (a) $\{- X_f: f \in \cF \}$ satisfies the $\beta$-Bernstein condition. If furthermore $\sup_{f \in \cF} {\bf E}[-X_f] < \infty$ then (b), $\{- X_f: f \in \cF \}$ also satisfies the $[0,\beta]$-Bernstein condition and also  $\sup_{f \in \cF} \E[X_f^2]$ is bounded, so that the family is regular. 
\end{enumerate}
\end{theorem}
 \begin{lemma}\label{lem:strongESIregularity}
 Let $\{X_f: f \in \cF \}$ be an ESI family. 
  \begin{enumerate}
      \item Suppose that the family is regular. Then {\bf C1} holds, with:\\ {\bf C1}: \ for each $k \geq 0$, for all  $0 < \delta \leq 1$,
 there exist  $C^*, C^{\circ}>0$ and  $\eta^{\circ}>0$ (that may depend on $\delta$) such that for all $0 < \eta \leq \eta^\circ$, all $f \in \cF$, $X_f \stochleq_{\eta^{\circ}} C^*$ and 
\begin{equation}\label{eq:strongesiregularity}
\E[|X_f|^{2+k} \exp(\eta X_f)] {\leq} C^{\circ}  (\E[X_f^2])^{1-\delta}
\end{equation}
\item 
Suppose that the witness-type condition (\ref{eq:tallball}) holds for the family $\{X_f: f \in \cF \}$ or for the family $\{ (X_f)_{-}: f \in \cF \}$. 
Then {\bf C2} holds, with: \\ {\bf C2}:\  there exist  $C >0$ and an $\eta^{\circ}>0$ such that for all $0 < \eta \leq \eta^\circ$, all $f \in \cF$,
\begin{equation}\label{eq:converseesiregularity}
\E[X_f^2] \leq C \cdot \E[X_f^2 \exp(\eta X_f)].
\end{equation}
 \end{enumerate}
\end{lemma}
To see how the above two results together  imply Theorem~\ref{thm:strongesi} in the main text, note that the implication (1) $\Rightarrow$ (2) in that theorem is a direct consequence of the fact that {\bf C1} in Lemma~\ref{lem:strongESIregularity} holds for regular ESI families for $\delta = 1- \beta$, any $0 \leq \beta < 1$, as expressed by Part 1 of that lemma, combined with Theorem~\ref{thm:strongesib}, Part 1. 

Implication (2) $\Rightarrow$ (3) of  Theorem~\ref{thm:strongesi} is trivial. Implication (3) $\Rightarrow$ (1) follows by  the fact that {\bf C2} in Lemma~\ref{lem:strongESIregularity} holds for families satisfying the witness-type condition, as expressed by Part 2 of that lemma, combined with Theorem~\ref{thm:strongesib}, Part 2.
\\ \ \\

\begin{proof}{\bf [of Lemma~\ref{lem:strongESIregularity}]}  \\ \noindent 
{\em Part 1}.
Let $s = \sup_{f \in \cF} \E[X_f^2]$ and set  $X= X_f$ for arbitrary $f \in \cF$.
Since we assume the family is regular, we can use Proposition~\ref{prop:newversionsubgammamain} to infer that there exists $\eta^*> 0$, $C^* > 0$ such that $X \stochleq_{\eta} C^*$ for all $0 < \eta \leq  \eta^*$.

For all $\eta' > 0, \eta > 0$,  all $0< \gamma < 2$ there must be constants $C, C'> 0$, such that for all $p> 0, q > 0$ with $1/p + 1/q=1$, it holds that :
\begin{align*}
& \E[|X|^{2+k} \exp(\eta X)] \\
&
 = \E[\ind{X \leq -1 } |X|^{2} \cdot \left( |X|^{k} \exp(\eta X) \right)] 
+ \E[\ind{-1 < X <  1 } |X|^{2+k} \exp(\eta X)] 
+ \E[\ind{X \geq 1} |X|^{2+k} \exp(\eta X)]  \\ &  \leq 
C' \E[\ind{X \leq -1 } X^{2}] +
\exp(\eta) \E[\ind{-1 < X <  1 } X^2] + 
\E[\ind{X\geq 1} |X|^{2-\gamma} \cdot \left( |X|^{k+ \gamma} \exp( \eta X) \right) ]
\\ & \leq (C' + \exp(\eta)) s (\E[X^2]/s)^{1-\delta} + 
C \E[\ind{X \geq 1} X^{2-\gamma}  \exp( (\eta' + \eta) X) ]
\\ & \leq 
(C'+ \exp(\eta) s^{\delta} \cdot (\E[X^2])^{1-\delta} + 
C \left( \E[(\ind{X \geq 1} X^{2-\gamma})^{p} ] \right)^{1/p} 
\left( \E[\exp(q (\eta'+ \eta) X) ] \right)^{1/q}, 
\end{align*}
where we in the first inequality we used that $|X|^ k \exp(\eta X)$ is bounded on $X \leq -1$ and in the second we used that $|X|^{k+\gamma} \exp(\eta X)$ is bounded by a constant times $\exp((\eta'+\eta) X)$ on  $X \geq 1$.
Then we used H\"older's inequality and the fact that $\E[X^2]/s \leq 1$. We now take $0 < \delta \leq 1$ as in the theorem statement and bound the second term further setting $1/p = 1-\delta$, $1/q = \delta$ and $\gamma = 2\delta$ (so that  $1/p+1/q=1$ as required and $(2-\gamma) p =2$):
\begin{align*}
& \left( \E[(\ind{X \geq 1} X^{2-\gamma})^{p} ]\right)^{1/p} 
\left( \E[\exp(q (\eta'+ \eta) X) ] \right)^{1/q} \\
& = \left( \E[(\ind{X \geq 1} X^{2-\gamma})^{p}] \right)^{1-\delta} 
\left( \E[\exp(\delta^{-1} (\eta'+ \eta) X) ] \right)^{\delta} \\
& \leq \left( \E[\ind{X \geq 1}  X^{2}] \right)^{1-\delta} 
\left( \E[\exp(\delta^{-1} (\eta'+ \eta) X) ] \right)^{\delta} \leq \left( \E[  X^{2}] \right)^{1-\delta}  C^*
\end{align*}
where the final equation follows for the specific choice  $\eta' =  \eta^*/(2 \delta)$ and any $\eta$ with $0 < \eta \leq \eta^{\circ} := \eta^*/(2\delta)$. 
Combining the two equations we find that for any such $\eta$, using that  the constants $C$ and $C^*$ do not depend on the $f$ with $X=X_f$, the result (\ref{eq:strongesiregularity}) follows. 
\\ \noindent
{\em Part 2}.
Let $X= X_f$ for arbitrary $f \in \cF$. First assume (\ref{eq:tallball}) for the family $\{X_f: f \in \cF\}$. We have:
\begin{align}\label{eq:variationa}
    & \E[X^2] = \E[X^2 \ind{X \geq 0}]
    + \E[X^2 \ind{X < 0 ; X^2 \leq C}] 
    + \E[X^2 \ind{X< 0; X^2 > C}] \\ \label{eq:variationb}
    & \leq \E[X^2 \exp(\eta X) \ind{X \geq 0}]
    + \E[X^2 \ind{X < 0 ; X^2 \leq C}] +
    \E[X^2 \ind{X^2 > C}] \\ \nonumber
    & \leq \E[X^2 \exp(\eta X) \ind{X \geq 0}]
    + \exp(\eta^{\circ} \sqrt{C} )\E[X^2 \ind{X < 0 ; X^2 \leq C}e^{\eta X}] +
    c \E[X^2]
\end{align}
for some $0 < c < 1$. Moving the rightmost term to the left-hand side and dividing both sides by $1-c$, we find that
$$
\E[X^2] \leq \frac{1}{1-c}\left( \exp(\eta^{\circ} \sqrt{C}) \E[X^2 \exp(\eta X)]
\right),
$$
which is what we had to prove. 
In case that (\ref{eq:tallball}) holds the family $\{(X_f)_{-}: f \in \cF\}$, the result follows by a minor variation of the above argument:
the rightmost term in (\ref{eq:variationa}) can then be bounded 
by $c \E[X_{-}^2]$, so the rightmost term in (\ref{eq:variationb}) can be replaced by $c \E[X_{-}^2]$, and then the final inequality still holds. 
\end{proof}
\\ \ \\ \noindent
\begin{proof}{\bf [of Theorem~\ref{thm:strongesib}]}
{\em Part 1}.
A short calculation using Jensen's inequality
shows that  for all $c, c^*, \eta >0$, we have
\begin{align}\label{eq:morningbound}
( X_f - c^* \cdot \E[X_f] + \eta c X_f^2
)^2 \leq 
 3  X_f^2 + 3 c^{*2} (\E[X_f])^2 +  3 \eta^2 c^2 X_f^4. 
\end{align}
Take $0 < c^*< 1$ and $\beta \in [0,b)$ as in the theorem statement. Set $\delta := 1-\beta$ and choose $\eta^{\circ}$ for this $\delta$ as in Lemma~\ref{lem:strongESIregularity}, Part 1 (which we can use because $0 < \delta \leq 1$). Without loss of generality let $\eta^{\circ} \leq 1$. We find for all $0 < \eta < \eta^{\circ}$, that for some $\eta'$ with $0 < \eta'< \eta^{\circ}$, and
for some constants $C_0, C_1, C_2, C_3, C_4, C^{\circ} > 0$, for $ \delta' = 2 \delta - \delta^2$ that 
\begin{align*}
&    \E[\exp(\eta ( X_f - c^* \cdot \E[X_f] + \eta c 
X_f^2
)]
\\
    & \leq 1 + \eta \E[ X_f - c^* \cdot \E[X_f] + \eta c 
    X_f^2 
    ]
    + \frac{1}{2} \eta^2 \E[( X_f - c^* \cdot \E[X_f] + \eta c 
    X_f^2
    )^2 e^{\eta'X_f}]  \\
    & \leq 1 + \eta (1-c^*) \cdot \E[X_f] + \eta^2 c \E[X_f^2]
+ \frac{3}{2}  \eta^2 \left(c{^*2} \E[X_f^2] \E[\rme^{\eta'X_f}]
    +  \E[X^2_f e^{\eta'X_f}] + 3 c^2 \E[X^4_f \rme^{\eta'X_f}] \right)\\
    & \leq 1 + \eta (1-c^*) \cdot \E[X_f] + \eta^2  c \E[X_f^2]
+ \frac{3}{2} c^{*2} \eta^2 C^* \E[X_f^2]
    + \frac{3}{2} (1 + c^2) \eta^2 C^{\circ} (\E[X_f^2])^{1-\delta} \\
    & \leq 1 + \eta (1-c^*) \cdot \E[X_f] + \eta^2   C_1 B \E[- X_f]^{1-\delta}
    + \eta^2 C_2 B^{1-\delta}  (\E[- X_f])^{(1-\delta)^2} \\
   & \leq 1 + \eta (1-c^*) \cdot \E[X_f] + \eta^2   C_3 \E[- X_f]^{1-\delta'}
    + \eta^2 C_4  (\E[- X_f])^{(1-\delta')} \\
    & \leq 1 + \eta \left(  (1-c^*) \cdot \E[X_f] + 
    \eta   \frac{C_3+C_4}{(1  - c^*)^{1-\delta'}} \left( (1 -c^*) \E[- X_f]\right)^{1-\delta'} \right) \\ &
    \leq 1 + \eta \left(  (1-c^*) \cdot \E[X_f] +   C^{\circ} \eta^{1/\delta} +  (1 - c^*)\E[- X_f]\right) \leq 1 + \eta \cdot C^{\circ} \eta^{1/\delta}
\leq e^{\eta C^{\circ} \eta^{1/\delta}}. \end{align*}
Here the first inequality is our extended Taylor approximation given by Lemma~\ref{lem:taylor}, stated and proved further below. For the second we used (\ref{eq:morningbound}). The third follows by Lemma~\ref{lem:strongESIregularity}, Part 1, applied with $k=0$ (for the first and second term within the rightmost brackets, and for determining $C^*$) and $k=2$, for the third term within those brackets, and with constant 
$C_0$ taken to be the maximum of the two corresponding constants $C^{\circ}$ in that lemma obtained with $k=0$ and $k=2$. The fourth follows from 
the $\beta$-Bernstein  condition (applied to the two final terms) for $\beta = 1-\delta$,  which holds by assumption. The fifth follows because
by Cauchy-Schwarz,  
$$\E[- X_f]^{1-\delta} =  \E[- X_f]^{1-\delta'}
\cdot  \E[X_f]^{\delta'- \delta} \leq 
\E[- X_f]^{1-\delta'} \cdot 
{\E[X_f^2]}^{(\delta'- \delta)/2}
$$
and the latter factor is bounded since we assume  the $\{X_f \}$ to be  regular. The sixth inequality above is just rearranging, and the seventh inequality is a `linearization' step. To see how it follows, note first that, 
for $p, q> 0 $ with  $1/p + 1/q = 1$, Young's inequality,
  $xy \leq |x|^p/p + |y|^q/q$, implies that for $0 < \beta < 1$, 
  \begin{equation}\label{eq:useful-young-inequality}
    ab^\beta \leq \frac{1-\beta}{\beta}(\beta a)^{\frac{1}{1-\beta}} + b
  \end{equation}
  which follows for $a,b>0$ by taking $\beta = 1/p$, $a = x^{p}$, and $b=y$. We apply this with $\beta = 1 -\delta'$, $b= (1 - c^*) \E[-X_f]$, $a = \eta (C_3+C_4)  / (1  - c^*)^{1-\delta'}$.
  The final inequality then gives the desired result.  

{\em Part 2\/} is similar to 1 but  much easier; we omit the details. 
\end{proof}
We end the section with the statement and proof of the validity of the second order Taylor approximation of the moment generating function, as used in the proof of Lemma~\ref{lem:strongESIregularity}.
 \begin{lemma}\label{lem:taylor}{\bf [``Extended Taylor'']}
  Suppose that  $\E[X^2] < \infty$ and also $\E[\rme^{\eta X}] < \infty$  for all $\eta \in [0,\eta_{\max}]$ and let $\eta^*$ be any number with $0 < \eta^* < \eta_{\max}$. Then for all $0 < \eta < \eta^*$ we have:
  \begin{equation}
      \E[\rme^{\eta X}] = 1 + \eta \E[X] + \frac{1}{2} \eta^2 \E[X^2 \rme^{\eta'X}]
  \end{equation}
  for some $\eta'$ with $\eta \leq \eta' < \eta^*$.
  \end{lemma}
  This is just the standard Taylor approximation of the moment generating function for random variable $X$ at $\eta = 0$. However, in this article we need this approximation also for the case that $\E[\rme^{\eta X}] = \infty$ for all $\eta < 0$ (e.g.~if $X$ has polynomial left tail), in which the standard Taylor's theorem does not apply any more, since the standard (two-sided) derivative at $\eta = 0$ is undefined. The lemma shows that nevertheless, everything still works as one would expect.  
   \begin{proof}
   Fix $\eta_0 \in (0,\eta^*)$. Then all derivatives of  $\E[\exp(\eta X)]$ exist at $\eta$ with $\eta_0 \leq \eta \leq \eta^*$, so that:
\begin{align}\label{eq:taylormade}
    \E[\rme^{\eta X}] = \E[\rme^{\eta_0 X}]  + (\eta - \eta_0) \E[X \rme^{\eta_0 X}] + \frac{1}{2} (\eta- \eta_0)^2 \E[X^2 \rme^{s(\eta_0,\eta) X}]
\end{align}
with $s$ some function $s:[0,\eta^*]^2 \rightarrow [\eta_0,\eta]$. Since 
\begin{align*}
    |\rme^{\eta X} X| \leq |\rme^{\eta X_+} X|
    \leq |\rme^{\eta^* X_+} X| \leq |X_-|
    + |\ind{X \geq 0} \rme^{\eta^* X_+} X|
    \leq |X_-| + |\rme^{\eta^* X} X|
\end{align*}
and we know  $\E[ |\rme^{\eta^* X} X|] < \infty$, we can use the dominated convergence theorem to conclude
that $\lim_{\eta_0 \downarrow 0} \E[X \rme^{\eta_0 X}] = \E[X]$. Analogously one shows that 
$\lim_{\eta_0 \downarrow 0} \E[X^2 \rme^{s(\eta_0,\eta) X}] = \E[X^2 \rme^{\eta'X}]$ for an $\eta' \in [\eta_0,\eta]$.
The result now follows by using these two limiting results in taking the limit for $\eta_0 \downarrow 0$ in (\ref{eq:taylormade}). 
   \end{proof}

\section{Proofs for Section~\ref{sec:random}}
\begin{proof}{ \bf (of Theorem~\ref{thm:ESIbothways})}
 Inequality \eqref{eq:juditha} 
 follows from Proposition \ref{prop:esi_characterization}, applied with $X = \hat\eta X_{\hat\eta}, Y = \hat\eta Y_{\hat\eta}, \eta =1$. \eqref{eq:partialconverse} follows from Proposition~\ref{prop:exponential_tail_esi}, with $X, Y $ and $\eta$ set in the same way (see the remark at the end of the proposition statement).
 
 Now we prove \eqref{eq:inexp}. Define the random variable $W_{\hat\eta} \coloneqq X_{\hat\eta} - Z_{\hat\eta}$, and for $(a,k)\in (0,\infty) \times \mathbb{N}$ define the event $\cE_{a,k} \coloneqq \left\{ a\cdot (k-1)\leq \hat\eta  W_{\hat\eta} \leq  a k\right\}$. With $Z=\hat\eta W_{\hat\eta}$, we will first show that
 \begin{align} \limsup_{a\downarrow 0}\sum_{k=1}^{\infty} a k\cdot P(\cE_{a,k}) \leq  \int_{0}^\infty P(Z\geq z) \rmd z. \label{eq:limitprob}
\end{align}
For $a,b>0$ and $k_{a,b} \coloneqq \floor{b/a}$, we have
\begin{align}
\sum_{k=1}^{k_{a,b}} a k \cdot P (\cE_{a,k})& = \sum_{k=1}^{k_{a,b}}  a k \cdot \left(P (Z > a \cdot (k-1)) - P (Z \geq a k) \right), \nonumber \\ &  \leq \sum_{k=1}^{k_{a,b}}  a k \cdot \left(P (Z \geq a \cdot (k-1)) - P (Z \geq a k) \right), \nonumber \\
& \leq \sum_{k=0}^{k_{a,b}-1} a \cdot P(Z \geq a k), \nonumber \\
& \leq a P(Z\geq 0)+\int_0^{k_{a,b}} a P(Z \geq a t) \rmd t, \quad (\text{$t\mapsto a\cdot P(Z\geq a t)$ is nonincreasing}) \nonumber \\
& = a P(Z\geq 0)+\int_0^{a k_{a,b}}  P(Z \geq z) \rmd z, \quad (\text{change of variable $y = a t$}) \nonumber  \\
& \leq a + \int_0^{b}  P(Z \geq z) \rmd z. \label{eq:lowertail}
\end{align}
Since \eqref{eq:lowertail} holds for all $a,b>0$, we have
\begin{align}
\limsup_{a\downarrow0} \sum_{k=1}^{\infty}a k\cdot P(\cE_{a,k}) = \limsup_{a\downarrow0} \left\{\sup_{b>0} \sum_{k=1}^{k_{a,b}} a k\cdot P(\cE_{a,k})\right\}  \stackrel{\eqref{eq:lowertail}}{\leq}  \int_{0}^\infty P(Z\geq z) \rmd z,
\end{align}
and thus, the desired inequality \eqref{eq:limitprob} follows. We now have, for all $a>0$,
\begin{align}
    \E \left[
    W_{\hat\eta}\right]
    & \leq
     \sum_{k=1}^{\infty}
      P\left(\cE_{a,k}\right) \cdot \E \left[\left.
      W_{\hat\eta}   \right\lvert \cE_{a,k} \right], \nonumber
    \\
    &\leq
      \sum_{k=1}^{\infty}  P\left(\cE_{a,k}\right)  \cdot \E \left[\frac{a k}{\hat\eta}\right]
       \quad (\text{by the definition of $\cE_{a,k}$})\nonumber  \\
       & = \left( \sum_{k=1}^{\infty}  P\left(\cE_{a,k}\right)  \cdot a k \right) \cdot \E \left[\frac{1}{\hat\eta}\right]. \nonumber \\
      \shortintertext{Since this inequality holds for all $a>0$, using \eqref{eq:limitprob} implies that}
\E\left[
    W_{\hat\eta}\right]
    & \leq  \E\left[\frac{1}{\hat\eta}\right] \cdot \int_0^{\infty}P (\hat\eta W_{\hat\eta}\geq t) \rmd t \nonumber  \\
      &\leq \E\left[\frac{1}{\hat\eta}\right]  \cdot \E \left[\rme^{\hat\eta W_{\hat\eta}} \right] \int_0^{\infty} \rme^{-t} \rmd t,  (\text{Markov's Inequality})\nonumber \\
      & \leq \E\left[\frac{1}{\hat\eta}\right],
\end{align}
where the last step follows from the fact that $W_{\hat\eta}\stochleq_{\hat\eta} 0$.
\end{proof}

\begin{proof}{\bf (of Proposition~\ref{prop:trans})}
Let's denote $X_{i,\eta}(Z;z^{n \m i})\coloneqq  X_{i,\eta}(z_1,\dots,z_{i-1},Z,z_{i+1},\dots,z_n)$, for all $\eta \in \cG$, $i\in [n]$, $z^{n\m i}\in \cZ^{n-1}$, and $Z\in \cZ$. In this way, \eqref{eq:assum} can be written as
\begin{align} X_{i,\eta}(Z_i; z^{n \m i}) \stochleq_{\eta} 0  ,\quad \text{for all $i\in [n]$ and $z^{n\m i}\in \cZ^{n-1}$.}  \end{align}
In particular, since this holds for all $z^{n\m i}\in \cZ^{n-1}$, we can also write \begin{align}\label{eq:chekpoint}X_{i,\eta}(Z_i; Z^{n \m i}) \stochleq_{\eta} 0.
\end{align}
Now let $Z_1^n,\dots,Z_n^n\in \cZ^n$ be $n$ i.i.d copies of $Z^n$. From \eqref{eq:chekpoint}, we have, for each $i\in [n]$, \begin{align}Y_{i,\eta}\coloneqq X_{i,\eta}(Z_{i,i}; Z_i^{n \m i}) \stochleq_{\eta} 0.\label{eq:interESI} \end{align} Since $(Y_{i,\eta})_{i\in [n]}$ are i.i.d, we can chain \eqref{eq:interESI}, for $i=1,\dots,n$, using Proposition \ref{prop:esi_sums} to get
\begin{align}
\sum_{i=1}^n Y_{i,\eta} \stochleq_{\eta} 0. \label{eq:secondCP}
\end{align}
By applying Proposition \ref{prop:randeta}, we have, for any random $\hat\eta$ in $\cG$,
\begin{align}
\E \left[\frac{\log |\cG|+1 }{\hat\eta}\right]& \geq \E \left[\sum_{i=1}^n Y_{i, \hat\eta}\right],\nonumber  \\
& = \E \left[\sum_{i=1}^n X_{i, \hat\eta}(Z_{i,i},Z^{n\m i}_i) \right],\nonumber   \\
& =\E \left[\sum_{i=1}^n X_{i,\hat\eta}(Z^n) \right],
\end{align}
where the last equality follows from the fact that $(Z_i^n)_{i\in[n]}$ are i.i.d copies of $Z^n$.
\end{proof}

\end{document}